\documentclass[a4paper,12pt,reqno]{article}

 \usepackage[english]{babel}
  \usepackage{amsmath,amsfonts,amssymb,amsthm}
   
   \usepackage{hyperref}
\usepackage{xcolor}
\hypersetup{
    colorlinks,
    linkcolor={red!80!black},
    citecolor={blue!70!black},
    urlcolor={blue!90!black}
}
    \usepackage{authblk}

\newtheorem{theorem}{Theorem}[section]

\newtheorem{lemma}[theorem]{Lemma}

\theoremstyle{definition}

\renewcommand{\leq}{\leqslant}
\renewcommand{\le}{\leqslant}
\renewcommand{\geq}{\geqslant}
\renewcommand{\ge}{\geqslant}

\def\R{\mathbb{R}}
\def\C{\mathbb{C}}

\let\e=\varepsilon

\let\.=\cdot
\let\0=\emptyset

\title{Decay estimates for evolution equations\\
with classical and fractional time-derivatives\thanks{Supported by
the Australian Research Council Discovery Project 170104880 NEW ``Nonlocal
Equations at Work''. The authors are members of INdAM/GNAMPA.}}

\author[*]{ Elisa Affili\thanks{Dipartimento di Matematica, Universit\`a degli studi di Milano,
Via Saldini 50, 20133 Milan, Italy, and 
Centre d'Analyse et de Math\'ematique Sociales,
\'Ecole des Hautes \'Etudes en Sciences Sociales,
54 Boulevard Raspail,
75006 Paris, France. {\tt elisa.affili@unimi.it}}\, and Enrico Valdinoci\thanks{Department of Mathematics and Statistics,
University of Western Australia,
35 Stirling Highway,
Crawley WA 6009, Australia, and
School of Mathematics and Statistics,
University of Melbourne,
813 Swanston Street, Parkville VIC 3010, Australia,
and Istituto di Matematica Applicata e Tecnologie Informatiche,
Consiglio Nazionale delle Ricerche,
Via Ferrata 1, 27100 Pavia, Italy,
and Dipartimento di Matematica, Universit\`a degli studi di Milano,
Via Saldini 50, 20133 Milan, Italy. {\tt enrico@mat.uniroma3.it} }}

\begin{document}

\maketitle

\begin{abstract}
	Using energy methods, we prove some power-law and exponential decay estimates for classical and nonlocal evolutionary equations. 
	The results obtained are framed into a general setting, which
	comprise, among the others,
	equations involving both standard and Caputo time-derivative, complex valued magnetic operators,
	fractional porous media equations and
	nonlocal Kirchhoff operators.
	
	Both local and fractional space diffusion are taken into account, possibly in a nonlinear setting.
	The different quantitative behaviors,
	which distinguish polynomial decays from exponential ones,
	depend heavily on the structure of the time-derivative involved in the equation.
\end{abstract}

\section{Introduction and main results}

\subsection{Setting of the problem}
Fractional calculus is becoming popular thanks to both the deep mathematics that it involves and its adaptability to the modelization of several real-world phenomena. As a matter of fact, integro-differential operators can describe nonlocal interactions of various type and
anomalous
diffusion by using suitable kernels or fractional time-derivatives, see e.g.  \cite{ultraslow}. 
Integro-differential equations and fractional derivatives have been involved in designing, for example, wave equations, magneto-thermoelastic heat conduction, hydrodynamics, quantum physics, porous medium equations.

A wide literature is devoted to the study of existence, uniqueness, regularity and asymptotic theorems.  Here we study the behaviour of the Lebesgue norm of solutions of integro-differential equations on bounded domains, extending the method of \cite{SD.EV.VV}
to a very broad class of nonlocal equations
and obtaining a power-law decay in time
of the $L^s$ norm with $s\geq 1$. 
Also, for the case of classical time-derivatives,
we obtain exponential decays in time. The difference between
polynomial and exponential decays in time is thus related to
the possible presence of a fractional derivative in the operator involving the time variable.

The setting in which we work
takes into account a
parabolic evolution of a function under the action of a spatial diffusive operator,
which possesses suitable ``ellipticity'' properties, can be either
classical or fractional, and can also be of nonlinear type.
We work in a very general framework that adapts to both local
and nonlocal operators. We comprise in this analysis also the case of
complex valued operators and of a combination
of fractional and classical time-derivatives.
\medskip

The main assumptions that we take is an ``abstract'' hypothesis
which extends a construction made in \cite{SD.EV.VV}, 
and which, roughly speaking, can be seen as a quantitative
counterpart of the uniform ellipticity of the spatial diffusive operators.
In~\cite{SD.EV.VV}, several time-decay estimates have been given
covering the cases in which the time-derivative is of fractional
type and the spatial operator is either the Laplacian,
the fractional Laplacian, the $p-$Laplacian and
the mean curvature equation. In this paper,
we deal with the cases in which {\em the time-derivative can be
either classical or fractional, or a convex combination of the two},
and we deal with new examples of spatial diffusive operators,
which include the case of a
complex valued operators. In particular,
we present applications to the {\em fractional porous medium equation}, 
to the {\em classical} and {\em fractional Kirchhoff equations}, to
the {\em classical} and {\em fractional magnetic operators}.\medskip

We recall that the Caputo derivative of order $\alpha\in(0,1)$ is given by
\begin{equation*}
\partial_t^\alpha u(t)  := \dfrac{d}{dt} \int_{0}^{t} \dfrac{u(\tau)-u(0)}{(t-\tau)^\alpha} d\tau
\end{equation*}
up to a normalizing constant (that we omit here for the sake of simplicity). 

Let also $\lambda_1, \lambda_2 \geq 0$ be fixed. We suppose, for concreteness,
that $$\lambda_1 + \lambda_2=1,$$
but up to a rescaling of the operator we can take $\lambda_1, \lambda_2$
any non negative number with positive sum. Let $\Omega \subset \R^n$ be a
bounded open set and let $u_0\in L^{\infty}(\R^n)$ such that $\text{supp} \,u_0 \subset \Omega$. Consider the Cauchy problem
\begin{equation} \label{sys:generalform}
\left\{ \begin{array}{lr}
(\lambda_1 \partial_t^{\alpha} + \lambda_2 \partial_t) [u] + \mathcal{N}[u]=0, & {\mbox{for all }}x\in \Omega, \ t>0, \\
u(x,t)=0, & {\mbox{for all }}x\in \R^n \setminus \Omega , \ t>0, \\
u(x,0)=u_0(x), & {\mbox{for all }}x\in \R^n ,
\end{array} \right.
\end{equation}
where $\mathcal{N}$ is a possibly nonlocal operator.
Given $s\in[1, +\infty)$ we want to find some estimates on the ${L}^s(\Omega)$
norm of $u$. To this end,
we exploit analytical techniques relying on
energy methods, exploiting also some 
tools 
that have been recently developed in \cite{KSVZ, VZ15, SD.EV.VV}.
Namely, as in \cite{SD.EV.VV}, we want to compare the $L^{s}$ norm of
the solution $u$ with an explicit function that has a power law decay,
and to do this we take advantadge of a suitable comparison result
and of the study of auxiliary fractional parabolic equations as
in \cite{KSVZ, VZ15}. 

\subsection{Notation and structural assumptions}
Let us recall that for a complex valued function $v:\Omega\to\C$ the Lebesgue norm is
\begin{equation*}
\Vert v \Vert_{L^s(\Omega)} = \left( \int_{\Omega} |v(x)|^s \; dx \right)^{\frac{1}{s}}
\end{equation*}
for any $s\in[1, +\infty)$. Also, we will call $\Re \{ z\}$ the real part of $z\in\C$.
The main assumption we take is the following: there exist $\gamma \in (0,+\infty) $ and $C\in (0,+\infty)$ such that
\begin{equation} \label{cond:complexstr}
\Vert u(\.,t) \Vert_{L^{s}(\Omega) }^{s-1+\gamma} \leq C \int_{\Omega} |u(x,t)|^{s-2} \Re \{ \bar{u}(x,t)\mathcal{N} [u](x,t)\} \; dx.
\end{equation}
The constants $\gamma$ and $C$ and their dependence from the parameters of the problem may vary from case to case. This structural assumption says, essentially, that $\mathcal{N}$ has
an elliptic structure and it is also related (via an integration by parts)
to a general form of the Sobolev inequality
(as it is apparent in the basic
case in which~$u$ is real valued, $s:=2$ and $\mathcal{N}u:=-\Delta u$). 

In our setting, the structural inequality in~\eqref{cond:complexstr}
will be the cornerstone to obtain general energy estimates,
which, combined with appropriate barriers, in turn
produce time-decay estimates. The results
obtained in this way are set in a general framework, and then we
make concrete examples of operators that satisfy the structural
assumptions, which is sufficient to establish asymptotic bounds that fit
to the different cases of interest and take into account the peculiarities
of each example in a quantitative way.

Our general result also includes Theorem 1 of \cite{SD.EV.VV}
as a particular case, since, if $\mathcal{N}$ and $u$ are real valued, the
\eqref{cond:complexstr} boils down to hypothesis (1.3) of \cite{SD.EV.VV} (in any case, the applications and examples covered here
go beyond the ones presented in \cite{SD.EV.VV} both for
complex and for real valued operators).

\subsection{Main results}

The ``abstract'' result that we establish here is the following:

\begin{theorem} \label{thm:complex}
	Let $u$ be a solution of the Cauchy problem \eqref{sys:generalform}, with $\mathcal{N}$ possibly complex
	valued. Suppose that there exist $s\in[1, +\infty)$, $\gamma\in(0,+\infty)$ and $C\in(0,+\infty)$ such that $u$ satisfies \eqref{cond:complexstr}.
	Then 
	\begin{equation} \label{claim1gen}
	(\lambda_1\partial_t^{\alpha} + \lambda_2\partial_t) \Vert u(\.,t) \Vert_{L^{s}(\Omega) } \leq -\dfrac{\Vert u(\.,t) \Vert_{L^{s}(\Omega) }^{\gamma}}{C},
	\qquad{\mbox{ for all }}t>0,\end{equation}
	where $C$ and $\gamma$ are the constants appearing in~\eqref{cond:complexstr}. 
	Furthermore,
	\begin{equation} \label{claim2gen}
	\Vert u(\.,t) \Vert_{L^{s}(\Omega) } \leq
	\dfrac{C_*}{1+t^{\frac{\alpha}{\gamma}}},\qquad{\mbox{ for all }}t>0,	
	\end{equation}
for some~$C_*>0$, depending only on $C$, $\gamma$, $\alpha$
	and $\Vert u_0(\.) \Vert_{L^{s}(\R^n)}$.
\end{theorem}

Theorem~\ref{thm:complex} here
comprises previous results in \cite{SD.EV.VV},
extending their applicability
to a wider class of equations, which include the
cases of both standard and fractional 
time-derivatives and complex valued operators.
 
We also recall that the power-law decay in~\eqref{claim2gen}
is due to
the behaviour of the solution of the equation 
\begin{equation} \label{mittagleffler}
	\partial_t^{\alpha} e(t)=-e(t),
\end{equation}
for $t\in(0, +\infty)$. Indeed, the solution
of~\eqref{mittagleffler} is explicit in terms of the Mittag-Leffler function and it is asymptotic to $\frac{1}{t^{\alpha}}$ as $t\rightarrow +\infty$ (see \cite{Mittag-Leffler}, \cite{Mittag-Leffler_asymt}); notice that
the latter decay corresponds to the one
in~\eqref{mittagleffler} when $\gamma=1$. 

As pointed out in \cite{KSVZ}, 
the power law decay for solutions of
time-fractional equations is, in general, unavoidable.
On the other hand, solutions of equations
driven by the standard time-derivative
of the type $$ \partial_t v(t) + \mathcal{N}[v](t)=0$$ often have a faster decay in many concrete examples, for instance for $\mathcal{N}=-\Delta$ where exponential decay is attained. This particular feature of
the classical heat equation is in fact a special
case of a general phenomenon, described
in details in the following result:

\begin{theorem}\label{thm:classic}
	Let $u$ be a solution of the Cauchy problem \eqref{sys:generalform} with only classical derivative ($\lambda_1=0$) and $\mathcal{N}$ possibly complex
	valued. Suppose that there exist $s\in[1, +\infty)$, $\gamma\in(0,+\infty)$ and $C\in(0,+\infty)$ such that $u$ satisfies \eqref{cond:complexstr}.
	Then, for some~$C_*>0$, depending only on the constants~$C$ and~$\gamma$
in~\eqref{cond:complexstr}, 
	and on~$\Vert u_0(\.) \Vert_{L^{s}(\R^n)}$, we have that:
	\begin{itemize}
		\item	if $0<\gamma \leq 1$ the solution $u$ satisfies
		\begin{equation} \label{claim3}
		\Vert u(\.,t) \Vert_{L^{s}(\Omega) } \leq
		C_* \, e^{-\frac{t}{C}},\qquad{\mbox{for all }}t>0;	
		\end{equation}
		\item if $ \gamma>1$, the solution $u$ satisfies
		\begin{equation} \label{claim4}
		\Vert u(\.,t) \Vert_{L^{s}(\Omega) } \leq
		\dfrac{C_*}{1+t^{\frac{1}{\gamma-1}}},\qquad{\mbox{for all }}t>0.	
		\end{equation}
	\end{itemize} 
\end{theorem} 

We stress that Theorem~\ref{thm:classic}
is valid for a very general class of diffusive
operators~${\mathcal{N}}$, including also
the ones which take into account
fractional derivatives in the space-variables.
In this sense, the phenomenon described in
Theorem~\ref{thm:classic} is that:
\begin{itemize} \item on the one hand,
the fractional behaviour induces power-law
decay,
\item on the other hand, for long times,
the interactions between different derivatives
``decouple'': for instance, a space-fractional
derivative, which would naturally induce
a polynomial decay, does not asymptotically ``interfere''
with a classical time-derivative
in the setting of Theorem~\ref{thm:classic},
and the final result is that the decay in time is
of exponential, rather than polynomial, type.
\end{itemize}

The fact that long-time asymptotics
of mixed type (i.e. classical time-derivatives
versus fractional-space diffusion) reflect the
exponential decay of linear ordinary differential
equations was also observed in~\cite{MR3703556}
for equations inspired by the Peierls-Nabarro model
for atom dislocations in crystal.\medskip

As we will see in the proof
of Theorem~\ref{thm:classic}, the idea is to find a supersolution of \eqref{claim1gen} and use a comparison principle in order to estimate the decay of the solution $u$. For the case of mixed derivatives, Vergara and Zacher \cite{VZ15} find both a supersolution and a subsolution decaying as $t^{-\frac{\alpha}{\gamma}}$. When $\alpha\rightarrow 1$, thus when the mixed derivative is approaching the classical one, the subsolution tends to 0. This allow possibly better decays, which are in fact proven. On the other side, the supersolution gains some extra decay, possibly reaching an exponential decay.
\medskip

The optimality of the decay estimates
obtained in our results
and some further comparisons with the existing literature are discussed in Subsection \ref{sss:comparison}. 

\subsection{Applications} \label{applications}

We now present several applications of Theorem \ref{thm:complex} to some concrete examples. 

\paragraph{The case of the fractional porous medium equation.}
Let $0<\sigma<1$
and 
\begin{equation}\label{kappa}
	K:\R^n \rightarrow \R^n
\end{equation}
 be the positive function
\begin{equation*}
	K(x):= c(n,\sigma) |x|^{-(n-2\sigma)},
\end{equation*}
being $c(n,\sigma)$ a constant. 
The fractional\footnote{As a matter of fact,
as clearly explained in 
\url{https://www.ma.utexas.edu/mediawiki/index.php/Nonlocal_porous_medium_equation},
the fractional porous medium equation
is ``the name currently given to two very different equations''.
The one introduced in~\cite{MR2737788}
has been studied in details in \cite{SD.EV.VV}
in terms of decay estimates. We focus here
on the equation introduced in~\cite{porous}.
As discussed in the above mentioned mediawiki page,
the two equations have very different structures and typically
exhibit different behaviors, so we think that it is a nice feature that,
combining the results here with those in~\cite{SD.EV.VV},
it follows that a complete set of decay estimates is valid for {\em both} the
fractional porous medium equations at the same time.}
porous medium operator (as defined in \cite{porous}) is 
\begin{equation} \label{op:porous}
	\mathcal{N}[u]:=-\nabla \cdot (u \nabla \mathcal{K}(u)), \qquad {\mbox{where}}\qquad\mathcal{K}(u):=u \star K
\end{equation}
where $\star$ denotes the convolution. This operator is used to describe the diffusion of a liquid under pressure in a porous environment in presence of memory effects and long-range
interactions, and also has some application in biological models, see~\cite{porous}.\medskip

In this framework, the following result holds:

\begin{theorem} \label{thm:porous}
	Take $u_0(x) \in L^{\infty}(\R^n)$ and let $u$ be a solution in $\Omega \times (0, + \infty)$ to \eqref{sys:generalform} with $\mathcal{N}$ the fractional porous medium operator as in \eqref{op:porous}. Then for all $s\in (1, +\infty)$ there exists $C_*>0$ depending on $n,\ s,\ \sigma,\ \Omega$ such that 
	\begin{equation*}
		 \Vert u(\cdot, t) \Vert_{L^{s}(\Omega) } \leq \dfrac{C_*}{1+t^{\alpha /2}}.
	\end{equation*}
	Also, in the case of only classical derivative ($\lambda_1=0$), we have 
		\begin{equation*}
		\Vert u(\cdot, t) \Vert_{L^{s}(\Omega) } \leq \dfrac{C_*}{1+t}
		\end{equation*}
	where $C_*>0$, possibly different than before,  depends on $n,\ s,\ \sigma,\ \Omega$.	
\end{theorem}

\paragraph{The case of the Kirchhoff operator and the fractional Kirchhoff operator.}
The Kirchhoff equation describes
the movement of an elastic string that is constrained at the extrema,
taking into account a possible growth of 
the tension of the vibrating string in view of its extension. It was first introduced by
Gustav Robert Kirchhoff in~1876, see
\url{https://archive.org/details/vorlesungenberm02kircgoog},
and fully addressed from the mathematical
point of view in the 20th century, see~\cite{MR0002699}.

Parabolic equations of Kirchhoff type have been widely studied during the '90s (see for example \cite{gobbino} and the reference therein). Recently a fractional counterpart to the Kirchhoff operator has been introduced by Fiscella and Valdinoci \cite{kirchhoff}.\medskip

The setting that we consider here is the following.
Let $m:[0,+\infty)\to[0,+\infty)$ be an nondecreasing function. A typical example is 
\begin{equation}\label{def:m}
	m(\xi)=m_0 +b\xi
\end{equation}
where $b> 0$ and $m_0 \geq 0$. We consider here both the cases\footnote{The case $m_0=0$ for \eqref{def:m} is usually called the degenerate case and it presents several additional
difficulties with respect with the non-degenerate case. } in which $m(0)>0$ and in which $m$ takes the form in~\eqref{def:m} with $m_0=0$. In this setting, the Kirchhoff operator that we take into account is 
\begin{equation}\label{KKOP} \mathcal{N}[u]:= m \left(\Vert \nabla u \Vert_{L^2(\Omega)}^2\right)  (-\Delta)u =0.  \end{equation}
Then, we obtain the following decay estimates:

\begin{theorem}\label{thm:cl_kirchhoff}
	Let $u$ be the solution of problem \eqref{sys:generalform} with $\mathcal{N}$ the Kirchhoff operator in \eqref{KKOP}. Then  
	there exist $\gamma>0$ and $C>0$ depending on $n,\ s,\ \Omega,\ \inf m(t)$ such that
	\begin{equation*}
	\Vert u(\.,t) \Vert_{L^{s}(\Omega) } \leq \dfrac{C}{1+t^{\frac{\alpha}{\gamma}}},\qquad{\mbox{for all }}t>0,
	\end{equation*}
	in the following cases:
	\begin{itemize}
		\item[(i)\;\;] for all $s\in[1, +\infty)$ when $m$ is non-degenerate; in particular, in this case~$\gamma=1$.
		\item[(ii)\;\;] for all $s\in[1,+ \infty)$ when $m$ is degenerate and $n\leq 4$; in particular, in this case~$\gamma=3$.
		\item[(iii)\;\;] for $s\leq\frac{2n}{n-4}$ when $m$ is degenerate and $n>4$; in particular, in this case~$\gamma=3$.
	\end{itemize}
	Moreover, if we take $\lambda_1=0$,
then there exists~$C_*>0$, $C'>0$
depending on $n,\ s,\ \Omega,\ \inf m(t)$, 
for which the following statements hold true:
\begin{itemize}
\item in case (i) we have
	 \begin{equation*}
	\Vert u(\cdot,t) \Vert_{L^{s}(\Omega) } \leq
	C_* \, e^{-\frac{t}{C'}},\qquad{\mbox{for all }}t>0,
	\end{equation*}
	 \item in cases (ii) and (iii) we have
	\begin{equation*}
	\Vert u(\cdot,t) \Vert_{L^{s}(\Omega) } \leq \dfrac{C_*}{1+t^\frac{1}{2}},\qquad{\mbox{for all }}t>0.
	\end{equation*}\end{itemize}
\end{theorem}

Next, we consider the case of the fractional Kirchhoff operator.
We take a nondecreasing positive function
$M:[0,+\infty)\rightarrow[0,+\infty)$. As for the classic Kirchhoff
operator, we consider either the case when $M(0)>0$
or the case $M(\xi)=b\xi$ with $b>0$.
We fix $\sigma\in(0,1)$. We define the norm 
\begin{equation}\label{FKPO-1}
	\Vert u(\., t) \Vert_{Z} = \left( \int_{\R^{2n}} \frac{|u(x,t)-u(y,t)|^2 }{|x-y|^{n+2\sigma}} \; dxdy \right)^{\frac{1}{2}} .
\end{equation}
Finally, the fractional Kirchhoff operator reads
\begin{equation}\label{FKPO}
	\mathcal{N}[u](x,t):= -M\left( \Vert u(\., t)\Vert_{Z}^2 \right) \int_{\R^n}  \frac{ u(x+y,t) + u(x-y,t) -2u(x,t)}{|x-y|^{n+2\sigma}} \; dy.
\end{equation}
In this setting, our result is the following:

\begin{theorem} \label{thm:Kirchhoff}
	Let $u$ be the solution of problem \eqref{sys:generalform} with $\mathcal{N}$ the fractional Kirchhoff operator in~\eqref{FKPO}. Then  
	 there exist $\gamma>0$ and $C>0$, depending on $K$, $n$, $s$, $\Omega$
and~$\inf M(\xi)$, such that
	\begin{equation*}
		\Vert u(\.,t) \Vert_{L^{s}(\Omega) } \leq
		\dfrac{C}{1+t^{\frac{\alpha}{\gamma}}}
,\qquad{\mbox{for all }}t>0,
	\end{equation*}
	in the following cases:
	\begin{itemize}
		\item[(i)\;\;] for all $s\in[1, +\infty)$ when $M$ is non-degenerate; in particular, in this case~$\gamma=1$.
		\item[(ii)\;\;] for all $s\in[1,+ \infty)$ when $M$ is degenerate and $n\leq 4\sigma$; in particular, in this case~$\gamma=3$.
		\item[(iii)\;\;] for $s\leq\frac{2n}{n-4\sigma}$ when $M$ is degenerate and $n>4\sigma$; in particular, in this case~$\gamma=3$.
	\end{itemize}
		Moreover, if we take $\lambda_1=0$,
then there exists $C_*>0$, depending on $n,\ s,\ \Omega,\ \inf M(t)$, such that:
\begin{itemize} \item in case (i) we have
	\begin{equation*}
	\Vert u(\cdot,t) \Vert_{L^{s}(\Omega) } \leq
	C_* \, e^{-\frac{t}{C'}}	,\qquad{\mbox{for all }}t>0,
	\end{equation*}
for some~$C'>0$, \item in cases (ii) and (iii) we have
	\begin{equation*}
	\Vert u(\cdot,t) \Vert_{L^{s}(\Omega) } \leq \dfrac{C_*}{1+t^\frac{1}{2}},\qquad{\mbox{for all }}t>0.
	\end{equation*}\end{itemize}
\end{theorem}

It is interesting to remark that
the cases~(i), (ii) and~(iii) in Theorem~\ref{thm:Kirchhoff} formally
reduce to those in Theorem~\ref{thm:cl_kirchhoff}
when~$\sigma\to1$.

\paragraph{The case of the magnetic operator and the fractional magnetic operator.}

We consider here an operator similar to Schr\"{o}dinger equation
with a magnetic potential (see e.g.~\cite{MR0142894} and the references therein), that is
\begin{equation}\label{NuMAG}
	\mathcal{N}[u]:= -(\nabla -iA)^2 u(x,t)= -\Delta u + |A|^2u -iA\cdot\nabla u -\nabla \cdot (iAu)
\end{equation} 
where $A: \R^n \rightarrow \R^n$ has the physical meaning of a magnetic field
(in this case, one usually studies the three-dimensional case~$n=3$, but our
approach is general).
The goal of these pages is to apply Theorem~\ref{thm:complex}
to the magnetic operator in~\eqref{NuMAG}, thus obtaining decay estimates in time in this framework.

It is interesting to remark that the operator in~\eqref{NuMAG}
is structurally very different from the linear Schr\"{o}dinger operator,
which corresponds to the choice
\begin{equation}\label{NuMAG:SC}
	\mathcal{N}[u]= -i(\Delta +V)u.
\end{equation}
Indeed, for the operator in~\eqref{NuMAG:SC} decay estimates 
in time do not\footnote{Indeed, if~$V\in\R$ and~$u$
is a solution of the Schr\"{o}dinger parabolic equation~$
\partial_t u+i(\Delta +V)u=0$ in~$\Omega$
with homogeneous data along~$\partial\Omega$, the conjugated equation reads~$
\partial_t \bar u-i(\Delta +V)\bar u=0$, and therefore
\begin{eqnarray*}&& \partial_t\int_\Omega |u(x,t)|^2\,dx=
\int_\Omega u(x,t)\,\partial_t\bar u(x,t)+\bar u(x,t)\,\partial_t u(x,t)\,dx\\&&\qquad
=i\int_\Omega u(x,t)\,\Delta\bar u(x,t)-\bar u(x,t)\,\Delta u(x,t)\,dx
\\&&\qquad=\int_\Omega \nabla\cdot\big( u(x,t)\,\nabla\bar u(x,t)-\bar u(x,t)\,\nabla u(x,t)\big)\,dx
=0,\end{eqnarray*}
where the last identity follows from the Divergence Theorem and the boundary conditions.
This shows that decay estimates in time are in general
not possible in this setting, thus highlighting an interesting
difference between the Schr\"{o}dinger operator in~\eqref{NuMAG:SC}
and the magnetic operator in~\eqref{NuMAG}.

This difference, as well as the computation above,
has a natural physical meaning,
since in the Schr\"odinger equation the squared modulus of
the solution represents
the probability density of a wave function,
whose total amount remains constant if no dissipative forces appear in the equation.} hold in general, not even in the case of classical 
time-derivatives.   \medskip

The decay estimate for the classical magnetic operator is the following:

\begin{theorem} \label{thm:cl_magnetic}
	Let $u$ be the solution of problem \eqref{sys:generalform} with $\mathcal{N}$ the magnetic operator in~\eqref{NuMAG}. 
Then for all $s \in [1, +\infty)$ there exist $C_1>0$ depending on $A$, $n$, $s$ and $\sigma$ such that
	\begin{equation*}
	\Vert u(\.,t) \Vert_{L^{s}(\Omega) } \leq \dfrac{C_1}{1+t^{{\alpha}}}\qquad{\mbox{for all }}t>0.
	\end{equation*}	
	Moreover, in the case of classical derivatives ($\lambda_1=0$), we have
	\begin{equation*}
			\Vert u(\.,t) \Vert_{L^{s}(\Omega) } \leq C_2\,e^{-\frac{t}{C_3}}\qquad{\mbox{for all }}t>0
	\end{equation*}
	for some $C_2$, $C_3>0$, depending on $A$, $n$, $ s$ and $\sigma$.
\end{theorem}

In \cite{groundstates} D'Avenia and Squassina introduced a fractional operator where a magnetic field $A: \R^n \rightarrow \R^n $ appears. Their aim was to study the behaviour of free particles interacting with a magnetic field. For a fixed $\sigma \in (0,1)$, such an operator in dimension $n$ reads
\begin{equation}\label{SQ}
	\mathcal{N}[u](x,t):= \int_{\R^n} \frac{u(x,t)-e^{i(x-y)A(\frac{x+y}{2})} u(y,t)}{|x-y|^{n+2\sigma}} \;dy.
\end{equation}
In the appropriate framework, the fractional magnetic operator in~\eqref{SQ}
recovers the classical magnetic operator in~\eqref{NuMAG}
as~$\sigma\to1$, see~\cite{MR3528339} (see also~\cite{MR3794886}
for a general approach involving also nonlinear operators).

In the setting of the fractional magnetic operator, we present the following result:

\begin{theorem} \label{thm:magnetic}
	Let $u$ be the solution of problem \eqref{sys:generalform} with $\mathcal{N}$ the
	fractional magnetic operator in~\eqref{SQ}. Then for all $s \in [1, +\infty)$ there exist $C_1>0$ depending on $n$, $s$ and $\sigma$ such that
\begin{equation*}
\Vert u(\.,t) \Vert_{L^{s}(\Omega) } \leq \dfrac{C_1}{1+t^{{\alpha}}}\qquad{\mbox{for all }}t>0.
\end{equation*}	
Moreover, in the case of classical derivatives ($\lambda_1=0$), we have
\begin{equation*}
\Vert u(\.,t) \Vert_{L^{s}(\Omega) } \leq C_2\,e^{-\frac{t}{C_3}}\qquad{\mbox{for all }}t>0,
\end{equation*}
for some $C_2$, $C_3>0$ depending on $n$, $s$ and $\sigma$.
\end{theorem}

The magnetic operators present a crucial difference with respect to
the other operators considered in the previous applications, since they are complex valued operators.\medskip

The examples provided here show that the ``abstract''
structural hypothesis \eqref{cond:complexstr} is reasonable and can be explicitly
checked in several cases of interest.
We are also confident that other interesting
examples fulfilling such an assumption
can be found, therefore Theorem \ref{thm:complex} turns out to play a pivotal
role in the asymptotics of real and complex valued,
possibly nonlinear, and possibly fractional, operators.

\subsubsection{Comparison with the existing literature} \label{sss:comparison}

In general, in problems of the type \eqref{sys:generalform} it is very difficult to provide
explicit solutions and often the system has no unique solution, see e.g.~\cite{biler3al}. Therefore, even partial information on the solutions is important. 

In the case of a Kirchhoff parabolic equation with purely classical time-derivative in the degenerate case $m(0)=0$, Ghisi and Gobbino \cite{gobbino} found the time-decay estimate
\begin{equation}\label{GOBLA1}  c (1+t)^{-1} \leq \Vert \nabla u(\cdot, t)
\Vert_{L^2(\Omega)}^2 \leq C (1+t)^{-1} \qquad{\mbox{for all }}t>0.\end{equation}
for some costants $C, c>0$ depending on initial data. {F}rom this,
performing an integration of the gradient along paths\footnote{More precisely,
the fact that~\eqref{GOBLA1} implies~\eqref{GOBLA2}
can be seen as a consequence of the following observation:
for every~$u\in C^\infty_0(\Omega)$,
\begin{equation}\label{L2GRAs}
\|u\|_{L^2(\Omega)}\le
C\,\|\nabla u\|_{L^2(\Omega)},
\end{equation}
where~$C>0$ depends on~$n$ and~$\Omega$.
Indeed, 
fix~$x_0\in\R^n$ such that~$B_1(x_0)\subset\R^n\setminus\Omega$
and~$\Omega\subset B_R(x_0)$, for some~$R>1$.
Then, for every~$x\in\Omega$ we have that~$|x-x_0|\in[1,R]$ and thus
\begin{eqnarray*}
&& |u(x)|^2=|u(x)-u(x_0)|^2=\left| 
\int_0^1 \nabla u(x_0+t(x-x_0))\cdot(x-x_0)\,dt
\right|^2\\
&&\qquad\le |x-x_0|^2\,\int_0^1 |\nabla u(x_0+t(x-x_0))|^2\,dt\le
R^2\,\int_0^1 |\nabla u(x_0+t(x-x_0))|^2\,dt.
\end{eqnarray*}
On the other hand, if~$t\in[0,1/R)$ we have that
$$ \big|t(x-x_0)\big|< \frac{|x-x_0|}{R}\le1$$
and so~$x_0+t(x-x_0)\in B_1(x_0)\subset\R^n\setminus\Omega$,
which in turn implies that~$\nabla u(x_0+t(x-x_0))=0$. This gives that
$$ |u(x)|^2\le R^2\,\int_{1/R}^1 |\nabla u(x_0+t(x-x_0))|^2\,dt.$$
Hence, using the substitution~$x\mapsto y:=x_0+t(x-x_0)$, we conclude that
\begin{eqnarray*}
&&\int_\Omega |u(x)|^2\,dx\le
R^2\,\int_{1/R}^1 \left[\int_{\R^n}|\nabla u(x_0+t(x-x_0))|^2\,dx\right]\,dt
=
R^2\,\int_{1/R}^1 \left[\int_{\R^n}|\nabla u(y)|^2\,\frac{dy}{t^n}\right]\,dt\\&&\qquad
\leq R^{n+2}\,\int_{1/R}^1 \left[\int_{\R^n}|\nabla u(y)|^2\,dy\right]\,dt
\leq
R^{n+2}\,\|\nabla u\|^2_{L^2(\R^n)}=R^{n+2}\,\|\nabla u\|^2_{L^2(\Omega)},\end{eqnarray*}
which proves~\eqref{L2GRAs}.}, one can find the estimate
\begin{equation}\label{GOBLA2}
\Vert u(\cdot, t) \Vert_{L^2(\Omega)} \leq C (1+t)^{-\frac{1}{2}}
\qquad{\mbox{for all }}t>0. \end{equation}
The latter is exactly the estimate we found
in Theorem~\ref{thm:cl_kirchhoff} as a particular case of our analysis.

The fractional porous medium equation with classical derivative has been studied 
by Biler, Karch and Imbert
in \cite{biler3al}, establishing some decay estimates of the $ L^s$ norm, such as
\begin{equation}\label{bilerdecay}
\Vert u(\cdot, t) \Vert_{L^{s}(\Omega) } \leq t^{-\frac{n}{n+2-2\sigma} \left(1-\frac{1}{s} \right)}.
\end{equation}
As a matter of fact, this decay
is slower than what we find in Theorem \ref{thm:porous}, which is asymptotic to $t^{-1}$
(in this sense, Theorem \ref{thm:porous} here
can be seen as an improvement of the estimates
in~\cite{biler3al}).

On the other hand, in~\cite{biler3al} the Authors also provide a {\em weak} solution that has exactly the decay in~\eqref{bilerdecay}, 
thus showing the optimality of~\eqref{bilerdecay}
in this generality,
while our result holds for {\em strong} solutions. Then, comparing Theorem \ref{thm:porous} here with the results in~\eqref{bilerdecay}
we obtain that a better decay is valid for regular solutions with respect to the one which is valid
also for irregular ones.

\section{Proofs}

This section contains the proofs of our main results. We start with the proof of Theorem \ref{thm:complex}.

In order to prove Theorem \ref{thm:complex}, we need a comparison result for the equation involving the mixed time-derivative. As a matter of fact,
comparison results
for the case of the Caputo derivative are available
in the literature, see e.g. Lemma 2.6 of \cite{VZ15}. In our arguments
we employ the differentiability of $u$ and the fact that $u$ is a strong solution, and we obtain:

\begin{lemma} \label{lemma:comparison}
	Let $T\in(0,+\infty)\cup\{+\infty\}$
and~$w, \ v: [0,T) \rightarrow [0,+\infty) $ be two Lipschitz continuous
functions.
Assume that~$w$ is a supersolution and $v$ is a subsolution at each differentiability point for the equation
	\begin{equation}\label{0919}
		\lambda_1 \partial_t^{\alpha} u(t) + \lambda_2 \partial_t u(t) =-ku^{\gamma}(t)
	\end{equation}
	with $\lambda_1$, $\lambda_2$, $\gamma$, $k >0$.

Then:
if \begin{equation}\label{X00}
w(0)> v(0), \end{equation}
we have that \begin{equation}\label{X01}
w(t)>v(t)\qquad{\mbox{ for all }}t\in(0,T).\end{equation}
\end{lemma}

\begin{proof}
	By contradiction, let us suppose that for some time $t \in(0,T)$ we have $w(t)=v(t)$, and let us call $\tau$ the first time for which the equality is reached. Then,
since~$w$ is a supersolution and $v$ is a subsolution of~\eqref{0919},
we obtain that
	\begin{equation}\label{QUA1}
		\lambda_1 \partial_t^{\alpha} (w-v)(\tau) + \lambda_2 \partial_t (w-v)(\tau) \geq -k [w^{\gamma}(\tau) - v^{\gamma}(\tau)]=0.
	\end{equation}
Now we distinguish two cases, depending on whether or not~$w-v$ is differentiable at~$\tau$.
To start with,	suppose that $w-v$ is differentiable at~$\tau$.
Since~$w\ge v$ in~$(0,\tau)$, we have that
\begin{equation*}
\partial_t (w-v)(\tau) \le 0.\end{equation*}
{F}rom this and~\eqref{QUA1}, we obtain that
\begin{eqnarray*}
0&\le&\partial_t^{\alpha} (w-v)(\tau) \\&=&
\frac{(w-v)(\tau) - (w-v)(0)}{\tau^{\alpha}} +\alpha \int_0^{\tau} \frac{(w-v)(\tau) - (w-v)(\rho)}{(\tau-\rho)^{1+\alpha}} d\rho\\&=&
-\frac{ (w-v)(0)}{\tau^{\alpha}} -\alpha \int_0^{\tau} \frac{ (w-v)(\rho)}{(\tau-\rho)^{1+\alpha}} d\rho\\&\le&
-\frac{ (w-v)(0)}{\tau^{\alpha}} 
.\end{eqnarray*}
This is in contradiction with~\eqref{X00} and so it proves~\eqref{X01} in this case.

Now we focus on the case in which~$w-v$ is not differentiable at~$\tau$.
Then, there exists a sequence~$t_j\in(0,\tau)$
such that $w-v$ is differentiable at~$t_j$, with~$\partial_t(w-v)(t_j)\le 0$
and~$t_j\to\tau$ as~$j\to+\infty$.
Consequently, 
since~$w$ is a supersolution and $v$ is a subsolution of~\eqref{0919},
we obtain that
	\begin{equation}\label{QUA2}
\begin{split}
&\frac{(w-v)(t_j) - (w-v)(0)}{t_j^{\alpha}} +\alpha \int_0^{t_j} \frac{(w-v)(t_j)
- (w-v)(\rho)}{(t_j-\rho)^{1+\alpha}} d\rho
\\ =\;&
\partial_t^{\alpha} (w-v)(t_j)
\\ \ge\;&
\partial_{t}^{\alpha} (w-v)(t_j) + \frac{\lambda_2}{\lambda_1} \partial_t (w-v)(t_j) 
\\ \geq\;& -\frac{k}{\lambda_1}\, [w^{\gamma}(t_j) - v^{\gamma}(t_j)]
.\end{split}\end{equation}
Now we observe that if~$f$ is a Lipschitz function and~$t_j\to\tau>0$
as~$j\to+\infty$, then
\begin{equation}\label{VITALI}
\lim_{j\to+\infty}
\int_0^{t_j} \frac{f(t_j)
- f(\rho)}{(t_j-\rho)^{1+\alpha}} d\rho
=\int_0^{\tau} \frac{f(\tau)
- f(\rho)}{(\tau-\rho)^{1+\alpha}} d\rho.
\end{equation}
To check this, let
$$ F_j(\rho):=\chi_{(0,t_j)}(\rho)\,
\frac{f(t_j)
- f(\rho)}{(t_j-\rho)^{1+\alpha}},$$
and let~$E\subset(0,+\infty)$ be a measurable set, with measure~$|E|$ less than
a given~$\delta>0$. Let also~$q:=\frac{1+\alpha}{2\alpha}>1$
and denote by~$p$ its conjugated exponent. Then, by H\"older inequality, for large~$j$
we have that
\begin{eqnarray*}
\int_E| F_j(\rho)|\,d\rho&\le& |E|^{1/p}\,\left( \int_0^{+\infty} |F_j(\rho)|^q\,d\rho\right)^{1/q}\\
&\le& \delta^{1/p}\,\left( \int_0^{t_j} 
\frac{|f(t_j) - f(\rho)|^q}{(t_j-\rho)^{(1+\alpha)q}}
\,d\rho\right)^{1/q}\\
&\le& L\,\delta^{1/p}\,\left( \int_0^{t_j} 
\frac{d\rho}{(t_j-\rho)^{\alpha q}}\right)^{1/q}\\
&=& L\,\delta^{1/p}\,\left( \int_0^{t_j} 
\frac{d\rho}{(t_j-\rho)^{(1+\alpha)/2}}\right)^{1/q}\\
&=& L\,\delta^{1/p}\,\left( \frac{2 t_j^{(1-\alpha)/2}}{1-\alpha}\right)^{1/q}
\\&\le&
L\,\left( \frac{2 (\tau+1)^{(1-\alpha)/2}}{1-\alpha}\right)^{1/q}
\,\delta^{1/p}
,\end{eqnarray*}
where~$L$ is the Lipschitz constant of~$f$. Consequently, by the
Vitali Convergence Theorem,
we obtain that
$$ \lim_{j\to+\infty}\int_0^{+\infty} F_j(\rho)\,d\rho=
\int_0^{+\infty} \lim_{j\to+\infty}F_j(\rho)\,d\rho,$$
which gives~\eqref{VITALI}, as desired.

Now, we take the limit as~$j\to+\infty$ in~\eqref{QUA2},
exploiting~\eqref{VITALI} and the fact that~$w(\tau)=v(\tau)$. In this
way, we have that
$$
-\frac{ (w-v)(0)}{\tau^{\alpha}} -\alpha \int_0^{\tau} \frac{
(w-v)(\rho)}{(\tau-\rho)^{1+\alpha}} d\rho
\ge0.$$
Since~$w\ge v$ in~$(0,\tau)$, the latter inequality implies that
$$
-\frac{ (w-v)(0)}{\tau^{\alpha}} 
\ge0.$$
This is in contradiction with~\eqref{X00} and so it completes the proof of~\eqref{X01}.
\end{proof}

It is also useful to observe that Lemma~\ref{lemma:comparison}
holds true also for the classical derivative (i.e. when~$\lambda_1=0$).
We give its statement and proof for the sake of completeness:

\begin{lemma} \label{lemma:comparison2}
Let $T\in(0,+\infty)\cup\{+\infty\}$,
$w, \ v: [0,T) \rightarrow [0,+\infty)$ be two Lipschitz continuous
functions.
Assume that~$w$ is a supersolution and $v$ is a subsolution at each differentiability point for the equation
	\begin{equation}\label{09192}
		\partial_t u(t) =-ku^{\gamma}(t)
	\end{equation}
	with $\gamma$, $k >0$.

Then:
if \begin{equation}\label{X002}
w(0)> v(0), \end{equation}
we have that \begin{equation}\label{X012}
w(t)>v(t)\qquad{\mbox{ for all }}t\in(0,T).\end{equation}
\end{lemma}

\begin{proof} Suppose that~\eqref{X012} is false. Then there exists~$\tau\in(0,T)$
such that~$w>v$ in~$(0,\tau)$ and
\begin{equation}\label{TAU0}w(\tau)=v(\tau).\end{equation}
We fix~$\e>0$, to be taken as small as we wish in the sequel,
and define
\begin{equation}\label{TAU3}
f(t):=w(t)-v(t)+\e\,(t-\tau).
\end{equation}
We observe that
$$f(0)=w(0)-v(0)-\e\tau\ge\frac{w(0)-v(0)}2>0,$$
as long as~$\e$ is sufficiently small,
and~$f(\tau)=w(\tau)-v(\tau)=0$.
Therefore there exists~$\tau_\e\in(0,\tau]$ such that
\begin{equation}\label{TAU2}
{\mbox{$f>0$ in~$(0,\tau_\e)$
and~$f(\tau_\e)=0$.}}\end{equation}
We claim that
\begin{equation}\label{TAU}
\lim_{\e\to0^+}\tau_\e=\tau.
\end{equation}
Indeed, suppose, by contradiction, that, up to a subsequence, $\tau_\e$
converges to some~$\tau_0\in[0,\tau)$ as~$\e\to0^+$. Then we have that
$$ 0=\lim_{\e\to0^+} f(\tau_\e)=\lim_{\e\to0^+}
w(\tau_\e)-v(\tau_\e)+\e\,(\tau_\e-\tau)=w(\tau_0)-v(\tau_0).$$
This is in contradiction with the definition of $\tau$
and so~\eqref{TAU} is proved.

Now, from~\eqref{TAU2}, we know that
there exists a sequence~$t_j\in(0,\tau_\e]$ such that~$f$
is differentiable at~$t_j$, $\partial_t f(t_j)\le0$
and~$t_j\to\tau_\e$ as~$j\to+\infty$.

Accordingly, we deduce from~\eqref{09192} and~\eqref{TAU3} that
$$ 0 \ge \partial_t f(t_j)=
\partial_t (w-v)(t_j)+\e
\ge-k\big( w^{\gamma}(t_j)-v^\gamma(t_j)\big)+\e.$$
Hence, taking the limit as~$j\to+\infty$,
\begin{equation}\label{TAY1}
\frac{\e}{k}\le w^{\gamma}(\tau_\e)-v^\gamma(\tau_\e)=
\big( v(\tau_\e)+\e\,(\tau-\tau_\e)\big)^\gamma-v^\gamma(\tau_\e).
\end{equation}
We claim that 
\begin{equation}\label{TAY2}\liminf_{\e\to0^+}
v(\tau_\e)>0.
\end{equation}
Indeed, if not, by~\eqref{TAU0} and~\eqref{TAU},
\begin{equation}\label{TAY3}
0=\liminf_{\e\to0^+} v(\tau_\e)=v(\tau)=w(\tau).
\end{equation}
We observe that this implies that
\begin{equation}\label{1GAMMA}
\gamma\in(0,1).
\end{equation}
Indeed, since~$w$ is a supersolution of~\eqref{09192},
we have that
\begin{eqnarray*}&&
w(t)\ge w(0)\,e^{-kt},\qquad{\mbox{when }}\gamma=1\\
{\mbox{and }}&&w(t)\ge\frac{1}{\left(
\frac1{w^{\gamma-1}(0)}+k(\gamma-1)t
\right)^{\frac1{\gamma-1}}},
\qquad{\mbox{when }}\gamma>1,\end{eqnarray*}
as long as~$w(t)>0$, and so for all~$t>0$.
In particular, we have that~$w(\tau)>0$, in contradiction with~\eqref{TAY3},
and this proves~\eqref{1GAMMA}.

Then, we use that~$v$ is a subsolution of~\eqref{09192} and~\eqref{TAY3}
to write that, for any~$t\in(0,\tau)$,
$$ 
-\frac{v^{1-\gamma}(t)}{1-\gamma}=
\frac{v^{1-\gamma}(\tau)-v^{1-\gamma}(t)}{1-\gamma}=\frac1{1-\gamma}
\int_t^\tau \partial_\rho (v^{1-\gamma}(\rho))\,d\rho
=\int_t^\tau \frac{\partial_t v(\rho)}{v^\gamma(\rho)}\,d\rho\le-k(\tau-t).$$
Therefore, recalling~\eqref{1GAMMA},
$$ v^{1-\gamma}(t)\ge k(1-\gamma)(\tau-t),$$
and thus
\begin{equation}\label{7ygfugv}
v(t)=v(t)\ge \big( k(1-\gamma)(\tau-t)\big)^{1/(1-\gamma)}.\end{equation}
Similarly, using that~$w$ is a supersolution of~\eqref{09192} and~\eqref{TAY3}
we obtain that, for any~$t\in(0,\tau)$,
$$ w(t)\le \big( k(1-\gamma)(\tau-t)\big)^{1/(1-\gamma)}.$$
Comparing this and~\eqref{7ygfugv}, we conclude that
$$ w(0)\le\big( k(1-\gamma)\tau\big)^{1/(1-\gamma)}\le v(0),$$
which is in contradiction with~\eqref{X002}, and so the proof of~\eqref{TAY2}
is complete.

Then, using~\eqref{TAY1}
and~\eqref{TAY2}, a Taylor expansion gives that
\begin{eqnarray*}
\frac{1}{k}&\le& \frac{v^\gamma(\tau_\e)}{\e}\,\left[
\left( 1+\frac{\e\,(\tau-\tau_\e)}{v(\tau_\e)}\right)^\gamma-1\right]
\\&=&\frac{
v^\gamma(\tau_\e)}{\e}\left[
\frac{\gamma\e\,(\tau-\tau_\e)}{v(\tau_\e)}+
O\left( \frac{\e^2\,(\tau-\tau_\e)^2}{v^2(\tau_\e)}\right)
\right]\\&=&
\frac{\gamma\,(\tau-\tau_\e)}{v^{1-\gamma}(\tau_\e)}+
O\left( \frac{\e\,(\tau-\tau_\e)^2}{v^{2-\gamma}(\tau_\e)}\right)
.\end{eqnarray*}
Then, sending~$\e\to0^+$ and recalling~\eqref{TAU}
and~\eqref{TAY2}, we conclude that~$\frac1k\le0$.
This is a contradiction and the proof of~\eqref{X012} is thereby complete.
\end{proof}

With this preliminary work, we are in the position of proving the general
claim stated in Theorem \ref{thm:complex}.

\begin{proof}[Proof of Theorem \ref{thm:complex}] 	
	First, notice that
	\begin{equation} \label{0721}
		{\partial_t |u|^{s}}= s |u|^{s-1} \left(\frac{\Re(u) \partial_{t} \Re(u)+ \Im(u) \partial_{t} \Im(u)}{|u|}  \right) = s|u|^{s-2} \Re\{\bar{u} \, \partial_t u\}.
	\end{equation}
	Using \eqref{0721} and exchanging the order of the integral and the derivative, we have
	\begin{equation}\label{ineq:complex0}
	\begin{split}
		\int_{\Omega} |u|^{s-2} \Re\{\bar{u} \, \partial_t u\} \; dx &= \int_{\Omega}  \frac{\partial_t |u|^{s}}{s} \; dx 
		=\frac{1}{s} \partial_t \int_{\Omega} |u|^s \; dx =  \frac{1}{s} \partial_t \Vert u(\cdot, t) \Vert_{L^{s} (\Omega) }^{s} \\
		& =\Vert u(\cdot, t) \Vert_{L^{s} (\Omega) }^{s-1} \partial_t \Vert u(\cdot, t) \Vert_{L^{s} (\Omega)}.
	\end{split}
	\end{equation}
Now we claim that
	\begin{equation}\label{ineq:complex}
	\Vert u(\.,t) \Vert_{L^{s}(\Omega) }^{s-1} \partial_t^{\alpha} (\Vert u(\.,t) \Vert_{L^{s}(\Omega) }) \leq \int_{\Omega} |u(x,t)|^{s-2} \Re \{ \bar{u}(x,t) \partial_t^{\alpha} (u(x,t)) \} \, dx.
	\end{equation}
	This formula is similar to one given in Corollary 3.1 of \cite{VZ15}
for general kernels. In our setting,
we provide
an easier proof for the case of the Caputo derivative, comprising also the case
of complex valued operators. 
To prove~\eqref{ineq:complex},
using the definition of Caputo derivative we see that
	\begin{equation*}
	\begin{split}
	\int_{\R^n} & |u(x,t)|^{s-2} \Re \{ \bar{u}(x,t)\partial_t^{\alpha} u(x,t) \} \; dx \\
	&=\int_{\Omega} |u(x,t)|^{s-2} \Re \left\{ \bar{u}(x,t) \left[ 
	\dfrac{u(x,t)-u(x,0)}{t^{\alpha}} + \alpha \int_{0}^{t} \dfrac{u(x,t)-u(x,\tau)}{(t-\tau)^{1+\alpha}} \;d\tau
	\right]
	\right\} \; dx \\
	&= \int_{\Omega} |u(x,t)|^{s-2} \bigg( \frac{|u(x,t)|^2 - \Re \{ \bar{u}(x,t) u(x,0) \} }{t^{\alpha}}  \\
	& \hspace{1em} + \alpha \int_{0}^{t} \frac{|u(x,t)|^2 - \Re \{ \bar{u}(x,t) u(x,\tau) \} }{(t-\tau)^{1+\alpha}} \; d\tau \bigg) dx.
	\end{split}
	\end{equation*}	
Hence, by using the H\"older inequality, we get
	\begin{equation*}
	\begin{split}	
	\int_{\R^n} & |u(x,t)|^{s-2} \Re \{ \bar{u}(x,t)\partial_t^{\alpha} u(x,t) \} \; dx \\
	& \geq \frac{ \Vert u(\cdot, t) \Vert_{L^s(\Omega)}^{s} - \Vert u(\cdot, t) \Vert_{L^s(\Omega)}^{s-1} \Vert u(\cdot, 0) \Vert_{L^s(\Omega)}  }{t^{\alpha}} +\alpha \int_{0}^{t} \frac{\Vert u(\cdot, t) \Vert_{L^s(\Omega)}^{s}}{(t-\tau)^{1+\alpha}} \; d\tau \\
	& \hspace{1em} - \alpha \int_{0}^{t} \dfrac{\Vert u(\cdot, t) \Vert_{L^s(\Omega)}^{s-1} \Vert u(\cdot, \tau) \Vert_{L^s(\Omega)}}{(t-\tau)^{1+\alpha}} \; d\tau \\
	& = \Vert u(\cdot, t) \Vert_{L^s(\Omega)}^{s-1} \bigg[ \dfrac{\Vert u(\cdot, t) \Vert_{L^s(\Omega)}-\Vert u(\cdot, 0) \Vert_{L^s(\Omega)}}{t^{\alpha}} \\
	& \hspace{1em} + \alpha \int_{0}^{t} \dfrac{\Vert u(\cdot, t) \Vert_{L^s(\Omega)} - \Vert u(\cdot, \tau) \Vert_{L^s(\Omega)}}{(t-\tau)^{1+\alpha}}\; d\tau  \bigg] \\
	& = \Vert u(\cdot, t) \Vert_{L^s(\Omega)}^{s-1} \partial_t^{\alpha} \Vert u(\cdot, t) \Vert_{L^s(\Omega)}.
	\end{split}
	\end{equation*}
This completes the proof of \eqref{ineq:complex}.
	
Now,	to make the notation simpler, we set $v(t):= \Vert u(\cdot, t) \Vert_{L^{s} (\Omega) }$. By combining \eqref{ineq:complex0} and~\eqref{ineq:complex}, we find that 
	\begin{equation*}
		v^{s-1}(t) \left( \lambda_1 \partial_t^{\alpha} v(t) + \lambda_2 \partial_t v(t) \right) \leq \int_{\Omega} |u|^{s-2}(x,t) \Re \left\{\bar{u}(x,t) \left( \lambda_1 \partial_t^{\alpha} u(x,t) +\lambda_2 \partial_t u(x,t) \right) \right\} dx
	\end{equation*}
	and so, using the fact that $u$ is a solution of~\eqref{sys:generalform}, we conclude that
	\begin{equation*}
		v^{s-1}(t) \left( \lambda_1 \partial_t^{\alpha} v(t) + \lambda_2 \partial_t v(t) \right) \leq - \int_{\Omega} |u|^{s-2}(x,t) \Re \{\bar{u}(x,t)  \mathcal{N}[u](x,t)\}  dx.
	\end{equation*}
{F}rom this, we use the structural hypothesis \eqref{cond:complexstr} and we obtain that
	\begin{equation*}
		v^{s-1}(t) \left( \lambda_1 \partial_t^{\alpha} v(t) + \lambda_2 \partial_t v(t) \right) \leq - \frac{v^{s-1+\gamma}(t)}{C}.
	\end{equation*}
	Hence, we have established
the claim in~\eqref{claim1gen} 
for all $t>0$ such that $v(t)>0$.
Then, suppose that for some $\bar{t}>0$ we have $v(\bar{t})=0$. Since~$v$ is nonnegative,
it follows that
\begin{equation}\label{210718}
\partial_t v(\bar{t})=0.
\end{equation}
On the other hand, if $v(t)=0$, then
\begin{equation}\label{1321}
{ \partial_t^{\alpha} v(t)
\le 0},\end{equation}
because
	\begin{equation*}
		\partial_t^{\alpha} v(t)= \frac{v(t)-v(0)}{t^{\alpha}} + \int_{0}^{t} \frac{v(t)-v(\tau)}{(t-\tau)^{1+\alpha}} d\tau \leq -\frac{v(0)}{t^{\alpha}} - \int_{0}^{t} \frac{v(\tau)}{(t-\tau)^{1+\alpha}} d\tau \leq 0.
	\end{equation*}  
So, by~\eqref{210718} and~\eqref{1321},
$\left( \lambda_1 \partial_t^{\alpha} v(\bar t) + \lambda_2 \partial_t v(\bar t) \right) \leq 0$, which gives~\eqref{claim1gen} also in this case, as desired.
	
Now we exhibit a supersolution $w(t)$ of the equation $(\lambda_1 \partial_t^{\alpha} + \lambda_2 \partial_t) v(t) = -\nu v^{\gamma}(t)$, where $\nu:=\frac{1}{C}$. 
For this, we recall
Section 7 of \cite{VZ15}, and we have that the function 
	\begin{equation*}
		w(t):= \left\{
		\begin{array}{ll}
		u_0 & {\mbox{if }} t\in [0,t_0], \\
		Kt^{-\frac{\alpha}{\gamma}} & {\mbox{if }}t\geq t_0, 
		\end{array}
		\right.
	\end{equation*}
	with $K:=u_0t_0^{\frac{\alpha}{\gamma}}$ is a supersolution of~$
\partial_t^{\alpha} w(t) = -\nu w^{\gamma}(t)$
as long as
	\begin{equation*}
		t_0 \geq \dfrac{u_0^{1-\gamma}}{\nu} \left(\frac{2^{\alpha}}{\Gamma(1-\alpha)} + \frac{\alpha}{\gamma} \frac{2^{\alpha + \frac{\alpha}{\gamma}}}{\Gamma(2-\alpha)}  \right).
	\end{equation*}
	We claim that $\partial_t w(t) \geq -\nu w^{\gamma} (t)$.
To prove this, it is equivalent to check that
	\begin{equation*}
		\frac{\alpha}{\gamma} u_0 \, t_0^{\frac{\alpha}{\gamma}} \, t^{-\frac{\alpha}{\gamma}-1} \leq \nu \, u_0^{\gamma} \, t_0^{{\alpha}} \, t^{-\alpha}
	,\end{equation*}
	which is in turn equivalent to
	\begin{equation*}
		\frac{\alpha}{\gamma \, \nu} u_0^{1-\gamma} \, t_0^{ \frac{\alpha}{\gamma} -\alpha} \leq t^{1+\frac{\alpha}{\gamma} -\alpha},
	\end{equation*}
and the latter equation holds if $$t_0 \geq \max \left\{ 1, \frac{\alpha}{\gamma  \nu} u_0^{1-\gamma} \right\}. $$ Therefore for $t_0$ big enough we have that $w(t)$ is a supersolution of the equation $(\lambda_1 \partial_t^{\alpha} + \lambda_2 \partial_t) v(t) = -\nu v^{\gamma}(t)$. Also, $w(t)$ satisfies
	\begin{equation*}
		w(t)\leq \frac{c}{1+t^{\frac{\alpha}{\gamma}}}
	\end{equation*}
	for some $c>0$ depending only on $\nu,\ \gamma, \ \alpha$ and $w(0)$. Hence by the comparison principle in Lemma \ref{lemma:comparison}, we infer that $v(t) \leq w(t)$, which
completes the proof of the desired result in~\eqref{claim2gen}.	
\end{proof}

\begin{proof}[Proof of Theorem \ref{thm:classic}]
	The proof is identical to the one of Theorem \ref{thm:complex} a
part from the construction of the supersolution
(and from the use of the comparison principle in Lemma~\ref{lemma:comparison2}
rather than in Lemma~\ref{lemma:comparison}). Our aim is now to find a supersolution to the equation \eqref{claim1gen} in the case $\lambda_1=0$, that we can write as
	\begin{equation}\label{CAS1}
		v'(t) = -\frac{1}{C}v^{\gamma}(t)
	\end{equation}
	where $C$ is the constant given in the hypothesis. To construct this supersolution,
we distinguish the cases~$0<\gamma \leq 1$
and~$\gamma>1$.
	
We define
\begin{equation}\label{79}
w_0:=\Vert u_0(\cdot) \Vert_{L^{s} (\Omega) },
\end{equation}

\begin{equation}
t_0:=\left\{\begin{matrix}
0 & {\mbox{if }}\gamma=1,\\ \max\left\{0, \ \frac{C}{1-\gamma}(w_0^{1-\gamma}-1) \right\}& {\mbox{if }}0<\gamma<1,
\end{matrix}
\right.\end{equation}
and 
\begin{equation}\label{theta1}
	\theta_0= \left(w_0-\dfrac{(1-\gamma)}{C}t_0\right).
\end{equation}
Notice that, for $0<\gamma<1$ 
\begin{equation}\label{theta}
	\theta_0 \leq 1.
\end{equation}
In fact, 
\begin{equation*}
	\frac{C}{1-\gamma} (w_0^{1-\gamma}-1) \leq t_0
\end{equation*}
implies
\begin{equation*}
	\left( w_0^{1-\gamma} - \frac{(1-\gamma)}{C}t_0  \right) \leq 1
\end{equation*}
and that proves \eqref{theta}.
Then, we see that the function
	\begin{equation}\label{97}
		w(t):= \left\{ 
		\begin{array}{lr}
		\left(w_0^{1-\gamma}-\dfrac{(1-\gamma)t}{C} \right)^{\frac{1}{1-\gamma}}, & {\mbox{if }}t\in[0,t_0] \\
		\theta_0  \,e^{\frac{t_0-t}{C}}, & {\mbox{if }}t\in(t_0, +\infty)
		\end{array}
		\right.
	\end{equation}
is a continuous and Lipschitz function, moreover it is a solution of~\eqref{CAS1}
in the case $\gamma=1$ and a supersolution of~\eqref{CAS1} in the case $0 <\gamma<1$. Indeed, to check this, we observe that, for $t\in[0, t_0]$,
\begin{eqnarray*}
	&& \hspace{-1em}  w'(t)+\frac1{C}w^{\gamma}(t)) \\
	&& \hspace{3em} = -\dfrac{1}{C}\left( w_0^{1-\gamma} -\dfrac{(1-\gamma)t}{C} \right)^{\frac{\gamma}{1-\gamma}}  + \dfrac1{C}\left( w_0^{1-\gamma} -\dfrac{(1-\gamma)t}{C} \right)^{\frac{\gamma}{1-\gamma}} \\
	&& \hspace{3em}  =0,
\end{eqnarray*}
while
for all~$t>t_0$,
\begin{eqnarray*}
&& C\left( w'(t)+\frac1{C}w^\gamma(t)\right)=
-\theta_0 e^{\frac{(t_0-t)}{C}}
+\theta_0^\gamma  e^{\frac{\gamma(t_0-t)}{C}}=
\theta_0^\gamma  e^{\frac{\gamma(t_0-t)}{C}}\left(1
-\theta_0^{1-\gamma} e^{\frac{(1-\gamma)(t_0-t)}{C}}\right)\\&&\qquad\ge
\theta_0^\gamma  e^{\frac{\gamma(t_0-t)}{C}}\left(1
-\theta_0^{1-\gamma} \right)\ge0,
\end{eqnarray*}
where the inequality holds thanks to \eqref{theta}. 
Notice also that the function $w$ is Lipschitz since it is piecewise continuous and derivable and it is continuous in the point $t=t_0$ because of the definition of $\theta$ given in \eqref{theta1}.
These observations establish
the desired supersolution properties for the
function in~\eqref{97} for~$0<\gamma\le1$.
{F}rom this and the comparison result
in Lemma \ref{lemma:comparison2}, used here with $w(t)$ and $v(t):= \Vert u(\cdot, t) \Vert_{L^{s} (\Omega) } $, we obtain that
$v(t)\le w(t)$ for any~$t\ge0$, and in particular,
\begin{equation}\label{99}
\Vert u(\cdot, t) \Vert_{L^{s} (\Omega) }\le
K e^{-\frac{t}{C}}
\qquad{\mbox{for any~$t>t_0$}}
\end{equation}
for $K:= \theta_0 e^{\frac{t_0}{C}}$.
This proves~\eqref{claim3}. 
\medskip

Now we deal with the case $\gamma>1$. In this case, we set $$ w_0:=\max \left\{\Vert u_0(\cdot) \Vert_{L^{s} (\Omega) }, \Big( \frac{C}{\gamma -1} \Big)^{\frac{1}{\gamma-1}}  \right\}. $$
	Then the function
	\begin{equation}\label{992}
	w(t):= \left\{ 
	\begin{array}{lr}
	w_0  , & {\mbox{if }}t\in[0,1] \\
	w_0  t^{-\frac{1}{\gamma-1}}, & {\mbox{if }}t>1
	\end{array}
	\right.
	\end{equation}
	is a supersolution of~\eqref{CAS1}. Indeed, if~$t>1$,
\begin{eqnarray*}
C\left( w'(t)+\frac1{C}w^\gamma(t)\right)=
-\frac{C}{\gamma-1} w_0 t^{-\frac{\gamma}{\gamma-1}}
+w_0^\gamma  t^{-\frac{\gamma}{\gamma-1}}=
w_0 t^{-\frac{\gamma}{\gamma-1}}\left(
w_0^{1-\gamma}-\frac{C}{\gamma-1}
\right)\ge0,
\end{eqnarray*}
while, if~$t\in(0,1)$,
$$ w'(t)+\frac1{C}w^\gamma(t)=\frac1{C}w^\gamma(t)\ge0.$$
This gives that the function in~\eqref{992} has the desired supersolution property and consequently we can apply the comparison result
in Lemma \ref{lemma:comparison2} with $w(t)$ and $v(t):= \Vert u(\cdot, t) \Vert_{L^{s} (\Omega) } $. In this way, we obtain that for all~$t\ge1$
$$ \Vert u(\cdot, t) \Vert_{L^{s} (\Omega) }\le 
w_0  t^{-\frac{1}{\gamma-1}},$$
and so
the proof of~\eqref{claim4}
is complete.
\end{proof}

Now, we present the applications of the abstract results to the operators introduced in Section \ref{applications}.

We start with the case of the fractional porous medium equation.

\begin{proof}[Proof of Theorem~\ref{thm:porous}]
In order to prove Theorem \ref{thm:porous}, our strategy is to verify the
validity of inequality \eqref{cond:complexstr} with~$\gamma:=2$
for the porous medium operator,
which would put us in the position of exploiting Theorems~\ref{thm:complex}
and~\ref{thm:classic}.

To this end, by elementary computations, up to changes of the positive constant $c$ depending on $n, \ s,$ and $ \sigma$, we see that
\begin{equation}\label{110}
	\begin{split}
	\int_{\Omega} u^{s-1}(x,t)\mathcal{N}[u](x,t) \; dx   &=\int_{\Omega} - u^{s-1} \nabla \cdot (u \nabla \mathcal{K} u)(x,t) \; dx \hspace{10em} \\
	&= \int_{\Omega} (s-1) u^{s-1}(x,t) \nabla u(x,t) \cdot \nabla \mathcal{K} u (x,t) dx \\
	&= \int_{\Omega} \nabla u^{s}(x,t) \cdot  \nabla \mathcal{K} u (x,t) \,dx
\end{split}
\end{equation}
Now, define for $\e>0$, the regularized operator
\begin{equation}
	\mathcal{K}_{\e}u= \int_{\Omega}c(n,\sigma) \frac{u(x-y,t)}{(|y|^2+\e^2)^{\frac{n-2\sigma}{2}}} dy.
\end{equation}
where $c(n,\sigma)$ is the same constant that appears in the definition of $\mathcal{K}$ in~\eqref{kappa}.
Notice that, since $u$ is regular, we have 
\begin{multline}\label{conv}
	\int_{\Omega} \nabla u^{s}(x,t) \cdot \nabla \mathcal{K}_{\e} u (x,t) \,dx \\
	 \leq \iint_{\R^n \times \R^n}
	 \frac{\chi_{\Omega}(x) \underset{x\in\Omega}{\sup} |\nabla u^s(x,t) | \, \chi_{\Omega}(x-y)\underset{(x-y)\in\Omega}{\sup} |\nabla u(x-y,t)|}{|y|^{n-2\sigma}} dxdy
\end{multline}
where $\chi$ is the characteristic function. Thus, thanks to~\eqref{conv} we can apply the
Dominated Convergence Theorem and obtain
\begin{equation} \label{limit}
	\underset{\e\rightarrow 0}{\lim} \int_{\Omega} \nabla u^{s}(x,t) \cdot  \nabla \mathcal{K}_{\e} u (x,t) \,dx  = \int_{\Omega} \nabla u^{s}(x,t) \cdot  \nabla \mathcal{K} u (x,t) \,dx.
\end{equation}
So, using~\eqref{110} and~\eqref{limit}, we have
\begin{equation}\label{lim}
\begin{split}
	\int_{\Omega} u^{s-1}(x,t)\mathcal{N}[u](x,t) \; dx   &=\underset{\e\rightarrow 0}{\lim} \int_{\Omega} \nabla u^{s}(x,t) \cdot  \nabla \mathcal{K}_{\e} u (x,t) \,dx \\
	&=\underset{\e\rightarrow 0}{\lim} \int_{\Omega} \nabla u^{s}(x,t) \cdot \int_{\Omega}  \frac{(-n+2\sigma)c(n,\sigma)u(y)}{(|x-y|^2+\e^2)^{\frac{n-2\sigma+2}{2}}} (x-y) dy  \,dx \\
	&=\underset{\e\rightarrow 0}{\lim} \iint_{\Omega\times\Omega}  \dfrac{c(n,\sigma) u(y,t)\nabla u^{s}(x,t) \cdot (y-x) }{(|x-y|^2+\e^2)^{\frac{n-2\sigma+2}{2}}} \; dy \,dx,
\end{split}	
\end{equation}
up to changes of the positive constant $c(n,\sigma)$.	
Now we adapt a method that was introduced in \cite{porous} to obtain $L^p$ estimates. We exchange the order of integration and have that
\begin{equation*}
	\begin{split}
	\iint_{\R^n}  c\, u(y,t) \dfrac{\nabla u^{s}(x,t) \cdot (y-x) }{(|x-y|^2+\e^2)^{\frac{n-2\sigma+2}{2}}} \; dx \,dy \\
	&  \hspace{-12em} = \iint_{\R^n}  c\, u(y,t) \dfrac{\nabla (u^{s}(x,t)-u^s(y,t)) \cdot (y-x) }{(|x-y|^2+\e^2)^{\frac{n-2\sigma+2}{2}}} \; dx \,dy \\
	&  \hspace{-12em} = \iint_{\R^n}  -c {(u^{s}(x,t)-u^s(y,t))u(y,t)} \Bigg[\dfrac{-n}{(|x-y|^2+\e^2)^{\frac{n-2\sigma+2}{2}}} \\ &
	\hspace{-9em }+\frac{(n-2\sigma+2)|x-y|^2}{(|x-y|^2+\e^2)^{\frac{n-2\sigma+4}{2}}}  \Bigg] dx \,dy \\
	&  \hspace{-12em} = \iint_{\R^n} c \frac{(u^{s}(x,t)-u^s(y,t))(u(x,t)-u(y,t))}2\Bigg[\dfrac{-n}{(|x-y|^2+\e^2)^{\frac{n-2\sigma+2}{2}}} \\ &
	\hspace{-9em }+\frac{(n-2\sigma+2)|x-y|^2}{(|x-y|^2+\e^2)^{\frac{n-2\sigma+4}{2}}}  \Bigg] dx \,dy.
	\end{split}
\end{equation*}
We observe now that, since $(u^{s}(x,t)-u^s(y,t))(u(x,t)-u(y,t))$ is always positive,  
\begin{equation*}
	\begin{split}
	& \hspace{-3em}\iint_{\R^n} c \frac{(u^{s}(x,t)-u^s(y,t))(u(x,t)-u(y,t))}2\Bigg[\dfrac{-n}{(|x-y|^2+\e^2)^{\frac{n-2\sigma+2}{2}}}  \\
	&+\frac{(n-2\sigma+2)|x-y|^2}{(|x-y|^2+\e^2)^{\frac{n-2\sigma+4}{2}}}  \Bigg] dx \,dy \\
	&  \leq \iint_{\R^n} c \frac{(u^{s}(x,t)-u^s(y,t))(u(x,t)-u(y,t))(2-2\sigma)}{2|x-y|^{n+2(1-\sigma)}} dx \,dy.
	\end{split}
\end{equation*}
Thus, again by the Dominated Convergence Theorem, we can pass to the limit in \eqref{lim} and obtain
\begin{equation}\label{111}
\begin{split}
	&\hspace{-1.5em}\int_{\Omega} u^{s-1}(x,t)\mathcal{N}[u](x,t) \; dx \\
	& \hspace{1.5em}=\iint_{\R^n} c \frac{(u^{s}(x,t)-u^s(y,t))(u(x,t)-u(y,t))(2-2\sigma)}{2|x-y|^{n+2(1-\sigma)}} dx \,dy.
\end{split}
\end{equation}
Now, we define $v(x,t)=u^{\frac{s+1}{2}}(x,t)$. Then, by inequality (2.15) of \cite{SD.EV.VV} we have, for some $C>0$,
\begin{equation*}
	C (u^s(x,t)-u^s(y,t))(u(x,t)-u(y,t) ) \geq |v(x,t)-v(y,t)|^2.
\end{equation*}
{F}rom this, \eqref{110} and~\eqref{111} we obtain that
\begin{equation}\label{112}
	\begin{split}&
C	\int_{\Omega} u^{s-1}(x,t)\mathcal{N}[u](x,t) \; dx  \\
&\qquad=
\iint_{\R^n} c\,C\, \dfrac{2-2s}{2} \dfrac{(u^{s}(x,t)-u^s(y,t))(u(x,t)-u(y,t))}{|x-y|^{n+2(1-\sigma)}} \;dx \,dy\\&\qquad
\ge
\iint_{\R^n} c\, \dfrac{2-2s}{2} \dfrac{|v(x,t)-v(y,t)|^2}{|x-y|^{n+2(1-\sigma)}} \;dx \,dy.\end{split}\end{equation}
Now we set $z:=(1-s)$; then $z\in(0,1)$ and $n\geq 2z$. Let also
$$p_z:= \dfrac{2n}{n-2z} \geq 2. $$
Then for any $q\in [2, p_z]$ we can apply the Gagliardo-Sobolev-{S}lobedetski\u\i \
fractionary inequality (compare \cite{guida}, Theorem 6.5) and obtain
\begin{equation}\label{113}
	\left( \int_{\Omega} u^{\frac{s+1}{2}q} \right)^{\frac{2}{q}}
	= \Vert v \Vert_{L^{q} (\Omega) }^2
	 \leq C \iint \dfrac{|v(x,t)-v(y,t)|^2}{|x-y|^{n+2z}} \; dxdy
\end{equation}
with $C$ depending only on $\Omega,\ n,\ z$ and $q$. 
In particular, choosing $q=2$, we deduce from~\eqref{113} that
\begin{equation}\label{091a}
\Vert u(\cdot, t) \Vert_{L^{s+1}(\Omega) }^{s+1}
	 \leq C \iint \dfrac{|v(x,t)-v(y,t)|^2}{|x-y|^{n+2z}} \; dxdy
\end{equation}
On the other hand, using the H\"{o}lder inequality, one has that
\begin{equation*}
	\Vert u(\cdot, t) \Vert_{L^{s}(\Omega) }^{s+1} \leq \Vert u(\cdot, t) \Vert_{L^{s+1}(\Omega)}^{s+1} |\Omega|^{1/s}.
\end{equation*}
Combining this and~\eqref{091a}, we obtain
\begin{equation*}
\Vert u(\cdot, t) \Vert_{L^{s+1}(\Omega) }^{s}
	 \leq C \iint \dfrac{|v(x,t)-v(y,t)|^2}{|x-y|^{n+2z}} \; dxdy,
\end{equation*}
up to renaming~$C>0$.
This and~\eqref{112} establish the validity
of \eqref{cond:complexstr} for $\gamma:=2$, as desired.
\end{proof}

Now we focus on the Kirchhoff equation, first dealing with the case
of classical derivatives.

\begin{proof}[Proof of Theorem \ref{thm:cl_kirchhoff}]
Our objective here is to verify the
validity of inequality \eqref{cond:complexstr} for suitable values
of~$\gamma$, and then make use of Theorems~\ref{thm:complex}
and~\ref{thm:classic}.

First we present the proof for the non-degenerate case, that takes place when $m(\xi)$ has a positive minimum. Let us call $m_0:=\min m(\xi)$, then
\begin{equation}\label{331}
	m \left(\Vert \nabla u \Vert_{L^2(\Omega)}\right) \int_{\Omega} |u|^{s-2}u (-\Delta)u \; dx \geq m_0	\int_{\Omega} |u|^{s-2}u (-\Delta)u \; dx.
\end{equation} 
In Theorem 1.2 of~\cite{SD.EV.VV}, the case of the Laplacian was considered:
there it was found that, for some $C>0$ depending on $s,\ n,\ \Omega$,
\begin{equation*}
	\int_{\Omega} |u|^{s-2}u (-\Delta)u \; dx \geq C \Vert u \Vert_{L^{s} (\Omega) }^s.
\end{equation*}
Combining this with~\eqref{331}
we see that \eqref{cond:complexstr} holds true for $\gamma=1$ and $C> 0$ depending on $s,\ n,\ \Omega, \ \min m(\xi)$. 

Now we deal with
the degenerate case, which requires the use of
finer estimates. In this case, we have that
\begin{equation}\label{0909}\begin{split}
	& \hspace{-3em} b \Vert \nabla u \Vert_{L^2(\Omega)}^2  \int_{\Omega} |u(x,t)|^{s-2} u(x,t)(-\Delta)u (x,t) \; dx\\
	& \hspace{3em} = b \Vert \nabla u \Vert_{L^2(\Omega)}^2 \int_{\Omega} |u(x,t)|^{s-2} |\nabla u (x,t)|^2 \; dx \\
	& \hspace{3em} \geq C \left(\int_{\Omega} |u(x,t)|^{\frac{s-2}{2}} |\nabla u(x,t)|^2 \; dx\right)^2
,\end{split}\end{equation}
where the first passage is an integration by parts and the last inequality holds in view of the Cauchy-Schwartz inequality. 

Now define \begin{equation}\label{992k}
v(x,t):=|u|^{\frac{s+2}{4}}(x,t).\end{equation} We have that
$$ |\nabla v|^2 = \left( \frac{s+2}{4} \right)^2 |u|^{\frac{s-2}{2}} |\nabla u|^2. $$
This and~\eqref{0909} give that
\begin{equation}\label{781-b}
\begin{split}
&\left( \frac{s+2}{4} \right)^4
b \Vert \nabla u \Vert_{L^2(\Omega)}^2  \int_{\Omega} |u(x,t)|^{s-2} u(x,t)(-\Delta)u (x,t) \; dx\\
\ge\,&
C \left(\int_{\Omega} \left( \frac{s+2}{4} \right)^2|u(x,t)|^{\frac{s-2}{2}} |\nabla u(x,t)|^2 \; dx\right)^2\\
=\,&
C \left(\int_{\Omega} |\nabla v(x,t)|^2 \; dx\right)^2.
\end{split}\end{equation}
We now use Sobolev injections (in the form
given, for instance, in formula (2.9) of~\cite{SD.EV.VV}), remembering that $v$ is zero outside $\Omega$. The inequality
\begin{equation}\label{781-a}
\Vert \nabla v \Vert_{L^2(\Omega)} \geq C \Vert v \Vert_{L^q(\Omega)} \end{equation}
holds 
\begin{equation}\label{PER781-a}
{\mbox{for all $q\geq 1$ if $n\in\{1, 2\}$,
and for all~$q\in\left[1,\displaystyle\frac{2n}{n-2}\right]$ if $n>2$.}}\end{equation}
Therefore, we set
\begin{equation}\label{PER781-b} q:=\frac{4s}{s+2}.\end{equation}
Recalling the ranges of~$s$ in claim~(iii) of Theorem~\ref{thm:cl_kirchhoff},
when~$n>2$ we have that
$$ (n-2) q-2n=\frac{4s(n-2)}{s+2}-2n=\frac{2}{s+2}\,\big(
(n-4)s-2n
\big)\le0,$$
which shows that the definition in~\eqref{PER781-b}
fulfills the conditions in~\eqref{PER781-a}, and so~\eqref{781-a}
is valid in this setting.

Hence, making use of~\eqref{992k}, \eqref{781-b} and \eqref{781-a}, up to renaming~$C$ line after line, we deduce that
\begin{eqnarray*}
&&b \Vert \nabla u \Vert_{L^2(\Omega)}^2  \int_{\Omega} |u(x,t)|^{s-2} u(x,t)(-\Delta)u (x,t) \; dx\\
&&\qquad\ge
C \|\nabla v(\cdot,t)\|^4_{L^2(\Omega)}
\ge C\|v\|_{L^q(\Omega)}^4= C\|u\|_{L^s(\Omega)}^{{s+2}}.
\end{eqnarray*}
These observations imply that
condition \eqref{cond:complexstr} is satisfied here
with $\gamma=3$ and $C$ depending on $s,\ m(\xi)$ and $\Omega$.
\end{proof}

Now we deal with the case of the fractional Kirchhoff equation.

\begin{proof}[Proof of Theorem~\ref{thm:Kirchhoff}]
As in the case of classical space-derivatives
dealt with in the proof of
Theorem~\ref{thm:cl_kirchhoff},
a quick proof for the non-degenerate case is available. Indeed,
\begin{equation*}
	\int_{\Omega} |u|^{s-2} u \mathcal{N}[u] \, dx =
	m \left(\Vert \nabla u\Vert_{L^{2} (\Omega) }^2 \right) \int_{\Omega} |u|^{s-2} u (-\Delta)^{\sigma}u \, dx  \geq \int_{\Omega} m_0 |u|^{s-2}u(-\Delta)^{\sigma}u \, dx
\end{equation*}
and in \cite{SD.EV.VV} it was shown that 
\begin{equation*}
	\int_{\Omega} m_0 |u|^{s-2}u(-\Delta)^{\sigma}u \, dx \geq \Vert u \Vert_{L^{s} (\Omega) }^s.
\end{equation*}
Thus, the validity of inequality \eqref{cond:complexstr} with~$\gamma=1$
is established in this case.
 	
We now deal with the degenerate case. 
We fix 
\begin{equation}\label{pge2si}
p\in[2, +\infty)\end{equation}
and we define
\begin{equation}\label{18}
r:= \frac{s+2}{2p}\qquad{\mbox{and}}\qquad
v(x,t):=|u(x,t)|^r.\end{equation} We claim that
\begin{equation} \label{kirch:claim1}
\begin{split}
&|v(x,t)-v(y,t)|^p  \\
& \hspace{2em}\leq c_0 |u(x,t)-u(y,t)|\sqrt{(u(x,t)-u(y,t)) (|u(x,t)|^{s-2}u(x,t) -|u(y,t)|^{s-2}u(y,t))}
\end{split}
\end{equation}
for some $c_0> 0$, independent of $u$. To prove this, we first observe
that the radicand in~\eqref{kirch:claim1} is well defined, since, for every~$a$, $b\in\R$ we have that
\begin{equation}\label{DO1}
(a-b) (|a|^{s-2}a -|b|^{s-2}b)\ge0.
\end{equation}
To check this, up to exchanging~$a$ and~$b$, we can suppose that~$a\ge b$.
Then, we have three cases to take into account: either~$a\ge b\ge0$,
or~$a\ge 0\ge b$, or~$0\ge a\ge b$.
If~$a\ge b\ge0$, we have that
$$ |a|^{s-2}a -|b|^{s-2}b= a^{s-1} -b^{s-1}\ge0,$$
and so~\eqref{DO1} holds true. If instead~$a\ge 0\ge b$, we have that
$$ |a|^{s-2}a -|b|^{s-2}b=|a|^{s-1} +|b|^{s-1}\ge0,$$
which gives~\eqref{DO1}
in this case. Finally, if~$0\ge a\ge b$,
$$ |a|^{s-2}a -|b|^{s-2}b=-|a|^{s-1} +|b|^{s-1}\ge0,$$
again since~$-|a|=a\ge b=-|b|$, thus completing the proof of~\eqref{DO1}.

Then, by~\eqref{DO1}, we have that~\eqref{kirch:claim1}
is equivalent to 
\begin{equation}\label{kirch:claimE}
	|v(x,t)-v(y,t)|^{2p} \leq c_1 (u(x,t)-u(y,t))^3{ (|u(x,t)|^{s-2}u(x,t) -|u(y,t)|^{s-2}u(y,t))}.
\end{equation}
We also note that when $u(x,t)=u(y,t)$ the inequality in~\eqref{kirch:claimE}
is trivially satisfied.
Hence, without loss of generality we can suppose that
\begin{equation}\label{WAG01la}
{\mbox{$|u(x,t)|>|u(y,t)|$,\; for fixed $x,\ y \in \R^n$.}}\end{equation}
We define
the function
\begin{equation}\label{EA2}
(-1,1)\ni\lambda\mapsto 	g (\lambda)=\frac{(1- |\lambda|^{\frac{s+2}{2p}})^{2p}}{(1-\lambda)^3(1-|\lambda|^{s-2}\lambda)}
\end{equation}
and we claim that
\begin{equation}\label{EA20}
\sup_{(-1,1)}g(\lambda) < +\infty.
\end{equation}
To this end, we point out that
$g$ is regular
for all $\lambda\in (-1,1)$, so, to establish~\eqref{EA20}, we only have to study the
limits of~$g$ for $\lambda \rightarrow -1^+$ and $\lambda \rightarrow 1^-$.

When~$\lambda \rightarrow -1^+$, this limit is immediate and $g(-1)=0$. On the other hand,
when~$\lambda \rightarrow 1^-$, we see that
\begin{equation*}
	\begin{split}
	\underset{\lambda\rightarrow 1^-}{\lim} g(\lambda) &= \underset{\varepsilon\rightarrow 0^+}{\lim} \frac{(1- (1-\e)^{\frac{s+2}{2p}})^{2p}}{(1-(1-\varepsilon))^3(1-(1-\varepsilon)^{s-1})} \\
	&= \underset{\varepsilon\rightarrow 0^+}{\lim} \frac{\left(
\frac{s+2}{2p}\varepsilon+O(\varepsilon^2)\right)^{2p}}{\varepsilon^3((s-1)\varepsilon + O(\varepsilon^2))} \\
	&=\underset{\varepsilon\rightarrow 0^+}{\lim} 
\frac{\varepsilon^{2p-4}\,\left(
\frac{s+2}{2p}+O(\varepsilon)\right)^{2p}}{(s-1 + O(\varepsilon))},
	\end{split}
\end{equation*}
which is finite, thanks to~\eqref{pge2si}.
Then~\eqref{EA20} holds true, as desired.

Then, using~\eqref{EA20} with~$\lambda:=\frac{b}{a}$, we have that 
\begin{equation}\label{WAG01la2}
\begin{split}
&{\mbox{for any~$a$, $b\in\R$ with $|a|>|b|$,}}\\
&\frac{\left(
\left|a\right|^{\frac{s+2}{2p}}
- \left|b\right|^{\frac{s+2}{2p}}\right)^{2p}}{
\left(a-b\right)^3\left(
\left|a\right|^{s-2}a
-\left|b\right|^{s-2}b\right)}
\;=\;
\frac{|a|^{s+2}}{|a|^{s-2}\;a^4}\cdot \frac{\left(1- \left|\frac{b}{a}\right|^{\frac{s+2}{2p}}\right)^{2p}}{
\left(1-\frac{b}{a}\right)^3\left(1-\left|\frac{b}{a}\right|^{s-2}
\frac{b}{a}\right)}\\&\qquad
\; =\;
\frac{(1- |\lambda|^{\frac{s+2}{2p}})^{2p}}{
(1-\lambda)^3(1-|\lambda|^{s-2}\lambda)}
\; =\;g(\lambda)\le C,
\end{split}\end{equation}
for some~$C>0$. Then, in view of~\eqref{WAG01la},
we can exploit~\eqref{WAG01la2}
with~$a:=u(x,t)$ and~$b:=u(y,t)$, from which we obtain that
\begin{eqnarray*}&& 
\left|
\left|u(x,t)\right|^{\frac{s+2}{2p}}
- \left|u(y,t)\right|^{\frac{s+2}{2p}}\right|^{2p}\;=\;
\left(
\left|u(x,t)\right|^{\frac{s+2}{2p}}
- \left|u(y,t)\right|^{\frac{s+2}{2p}}\right)^{2p}\\ &&\qquad\;\leq \;C\,
\left(u(x,t)-u(y,t)\right)^3\left(
\left|u(x,t)\right|^{s-2}u(x,t)
-\left|u(y,t)\right|^{s-2}u(y,t)\right).\end{eqnarray*}
This and~\eqref{18} 
imply~\eqref{kirch:claim1}, as desired.

Now, fixed
$p$ as in~\eqref{pge2si},
we set
\begin{equation}\label{8iswdjc8383}
z:=\frac{ 2\sigma}{p}\in(0,\sigma]\subset(0,1).\end{equation}
We apply the Gagliardo-Sobolev-{S}lobedetski\u\i \
fractional immersion (for instance, in the version given
in formula~(2.18) of~\cite{SD.EV.VV}) to $v$. 
In this way,
\begin{equation}\label{202020}
{\mbox{for all $q\in[1,+\infty)$ when $n\le zp$, and for all $q\in\left[1,
\displaystyle\dfrac{np}{n-zp}\right]$ when~$n>zp$,}}
\end{equation}
we have that
\begin{equation}\label{20}\begin{split}
\Vert u(\cdot,t) \Vert_{L^{\frac{(s+2)q}{2p}}(\Omega)}^{\frac{s+2}2}=
	\Vert v(\cdot,t) \Vert_{L^q(\Omega)}^p &\,\leq C \iint_{\R^{2n}}
	\frac{|v(x,t)-v(y,t)|^p}{|x-y|^{n+zp}} dxdy\\&=
	C \iint_{\R^{2n}}
	\frac{|v(x,t)-v(y,t)|^p}{|x-y|^{n+2\sigma}} dxdy,\end{split}
\end{equation}
where the first equality comes from~\eqref{18} and the
latter equality is a consequence of~\eqref{8iswdjc8383}.

Now we choose
\begin{equation}\label{LAq}
p:=\max\left\{ 2,\,\frac{s+2}{2}\right\}\qquad{\mbox{and}}\qquad
 q:=\frac{2sp}{s+2}.\end{equation}
Notice that condition~\eqref{pge2si} is fulfilled in this setting.
Furthermore, recalling~\eqref{8iswdjc8383} and
the assumptions in point~(iii)
of Theorem~\ref{thm:Kirchhoff}, we have that, when~$n>2\sigma=zp$,
we have
\begin{eqnarray*}
&&(n-zp)q-np=
\frac{2(n-2\sigma)sp}{s+2}-np=
\frac{p}{s+2}\,\big(
2(n-2\sigma)s-n(s+2)
\big)\\
&&\qquad=
\frac{p}{s+2}\,\big(
s(n-4\sigma)-2n
\big)\le0.
\end{eqnarray*}
As a consequence, we have that condition~\eqref{202020}
is fulfilled the setting prescribed by~\eqref{LAq},
hence we can exploit~\eqref{20} in this framework.

Then, from~\eqref{LAq} we have that
$$ \frac{(s+2)q}{2p}=s,$$
and so~\eqref{20} gives that
$$ \Vert u(\cdot,t) \Vert_{L^{s}(\Omega)}^{\frac{s+2}2}\le
	C \iint_{\R^{2n}}
	\frac{|v(x,t)-v(y,t)|^p}{|x-y|^{n+2\sigma}} dxdy.$$
Hence, recalling~\eqref{kirch:claim1}, up to renaming~$C>0$,
we have that
\begin{equation}\label{COM234520Aiekwd}
\begin{split} &\Vert u(\cdot,t) \Vert_{L^{s}(\Omega)}^{s+2}\\ \le\,&
	C\left( \iint_{\R^{2n}}
	\frac{|u(x,t)-u(y,t)|\sqrt{(u(x,t)-u(y,t))
	(|u(x,t)|^{s-2}u(x,t) -|u(y,t)|^{s-2}u(y,t))}}{|x-y|^{n+2\sigma}} dxdy\right)^2\\
	\le\,& C
\iint_{\R^{2n}}
	\frac{|u(x,t)-u(y,t)|^2}{|x-y|^{n+2\sigma}} dxdy\\
	&\times
\iint_{\R^{2n}}
	\frac{{(u(x,t)-u(y,t))
	(|u(x,t)|^{s-2}u(x,t) -|u(y,t)|^{s-2}u(y,t))}}{|x-y|^{n+2\sigma}} dxdy.
	\end{split}\end{equation}	
Notice also that, in the degenerate case, we deduce from~\eqref{FKPO-1}
and~\eqref{FKPO} that
\begin{equation}\label{ghUAJ:a9ok01}
\begin{split}
&\int_{\R^n} {\mathcal{N}}[u](x,t)\,|u(x,t)|^{s-2}\,u(x,t)\,dx\\
=\;&
-M_u\,
\iint_{\R^{2n}} 
\Big( u(x+y,t) + u(x-y,t) -2u(x,t) \Big) 
\,|u(x,t)|^{s-2}\,u(x,t)\,\frac{dx\,dy}{|y|^{n+2\sigma}}\\
=\;&
-2M_u\,\iint_{\R^{2n}} 
\big( u(y,t)-u(x,t) \big) 
\,|u(x,t)|^{s-2}\,u(x,t)\,\frac{dx\,dy}{|x-y|^{n+2\sigma}}\\
=\;&
2M_u\,\iint_{\R^{2n}} 
\big( u(x,t)-u(y,t) \big) 
\,|u(x,t)|^{s-2}\,u(x,t)\,\frac{dx\,dy}{|x-y|^{n+2\sigma}}\\
=\;&
M_u\,\iint_{\R^{2n}} 
\big( u(x,t)-u(y,t) \big) 
\,\big(|u(x,t)|^{s-2}\,u(x,t)-|u(y,t)|^{s-2}\,u(y,t)\big)\,\frac{dx\,dy}{|x-y|^{n+2\sigma}}
,\end{split}\end{equation}
with
\begin{equation}\label{ghUAJ:a9ok02}\begin{split} M_u\,&:=
M
\left( \int_{\R^{2n}}\frac{ |u(x,t)-u(y,t)|^2 }{|x-y|^{n+2\sigma}}
\,dx\,dy \right)\\&\ge
b\,\int_{\R^{2n}}\frac{ |u(x,t)-u(y,t)|^2 }{|x-y|^{n+2\sigma}}\,dx\,dy,
\end{split}\end{equation}
with~$b>0$.
	
Then, from~\eqref{ghUAJ:a9ok01}
and~\eqref{ghUAJ:a9ok02},
\begin{eqnarray*}
&&\int_{\R^n} {\mathcal{N}}[u](x,t)\,|u(x,t)|^{s-2}\,u(x,t)\,dx
\ge b\,\int_{\R^{2n}}\frac{ |u(x,t)-u(y,t)|^2 }{|x-y|^{n+2\sigma}}\,dx\,dy\\
\qquad&&\times
\iint_{\R^{2n}} 
\big( u(x,t)-u(y,t) \big) 
\,\big(|u(x,t)|^{s-2}\,u(x,t)-|u(y,t)|^{s-2}\,u(y,t)\big)\,\frac{dx\,dy}{|x-y|^{n+2\sigma}}.
\end{eqnarray*}
Comparing this with~\eqref{COM234520Aiekwd}, we conclude that
$$ \Vert u(\cdot,t) \Vert_{L^{s}(\Omega)}^{s+2}\le C\int_{\R^n} {\mathcal{N}}[u](x,t)\,|u(x,t)|^{s-2}\,u(x,t)\,dx,$$
up to renaming~$C$.
This gives that hypothesis \eqref{cond:complexstr} is fulfilled
in this case with $\gamma=3$.
\end{proof} 


Now we deal with the case of the magnetic operators.
We start with the case of classical space-derivatives.
For this, we exploit an elementary, but useful, inequality,
stated in the following auxiliary result:

\begin{lemma}
Let~$a$, $b\in\R$, and~$\alpha$, $\beta$, $t\in\R^n$.
Then
\begin{equation}\label{ST:00}
(a^2+b^2)\Big(|a t-\beta|^2 + |bt+\alpha|^2 \Big)\ge 
|a\alpha+b\beta|^2.\end{equation}
\end{lemma}

\begin{proof} For any~$t\in\R^n$, we define
\begin{equation} \label{872wj2xz}
f(t):=
(a^2+b^2)\Big(|a t-\beta|^2 + |bt+\alpha|^2 \Big)- 
|a\alpha+b\beta|^2.\end{equation}
We observe that
\begin{equation}\label{ST:01}
\begin{split}
f(0)\,&=
(a^2+b^2)(\alpha^2+\beta^2) - |a\alpha+b\beta|^2\\ &=
a^2\alpha^2+
a^2\beta^2+
b^2\alpha^2+
b^2\beta^2
- ( a^2\alpha^2+b^2\beta^2 +2ab\alpha\beta)
\\ &=
a^2\beta^2+
b^2\alpha^2
- 2ab\alpha\beta\\
&= |a\beta-b\alpha|^2.
\end{split}\end{equation}
Moreover
\begin{equation}\label{ST:02}
\lim_{|t|\to+\infty} f(t)=\left\{
\begin{matrix}
+\infty & {\mbox{ if }} a^2+b^2>0,\\
0 & {\mbox{ otherwise.}}
\end{matrix}
\right.
\end{equation}
Now we claim that
\begin{equation}\label{ST:03}
f(t)\ge0,\end{equation}
for all~$t\in\R^n$. To prove~\eqref{ST:03} we argue by contradiction
and assume that
$$ \inf_{\R^n} f<0.$$
Then, in view of~\eqref{ST:01} and~\eqref{ST:02}, we have that
\begin{equation}\label{ST:04}
f(\bar t)=\inf_{\R^n} f<0,\end{equation}
for some~$\bar t\in \R^n$. As a consequence,
$$ 0=\nabla f(\bar t)=
2(a^2+b^2)\Big(a(a \bar t-\beta) + b(b\bar t+\alpha) \Big)=
2(a^2+b^2)\Big((a^2+b^2) \bar t-a\beta+b\alpha \Big),$$
which implies that
$$ \bar t=
\frac{a\beta-b\alpha}{a^2+b^2}.$$
Thus, we substitute this information into~\eqref{872wj2xz}
and we obtain that
\begin{eqnarray*}
f(\bar t) &=&
(a^2+b^2)\left(\left|
\frac{a^2\beta-ab\alpha}{a^2+b^2}-\beta\right|^2 + 
\left|\frac{ab\beta-b^2\alpha}{a^2+b^2}+\alpha\right|^2 \right)- 
|a\alpha+b\beta|^2\\&=&
(a^2+b^2)\left(\left|
\frac{b^2\beta+ab\alpha}{a^2+b^2}\right|^2 + 
\left|\frac{ab\beta+a^2\alpha}{a^2+b^2}\right|^2 \right)- 
|a\alpha+b\beta|^2\\&=&
(a^2+b^2)\left(b^2\left|
\frac{b\beta+a\alpha}{a^2+b^2}\right|^2 + 
a^2\left|\frac{b\beta+a\alpha}{a^2+b^2}\right|^2 \right)- 
|a\alpha+b\beta|^2\\&=&
(a^2+b^2)(a^2+b^2)\left|
\frac{b\beta+a\alpha}{a^2+b^2}\right|^2 - 
|a\alpha+b\beta|^2
\\&=&0.\end{eqnarray*}
This is in contradiction with~\eqref{ST:04} and so it proves~\eqref{ST:03},
which in turn imples~\eqref{ST:00}, as desired.
\end{proof}

With this, we are now in the position of completing the proof
of Theorem \ref{thm:cl_magnetic} and obtain the desired decay estimates
for the classical magnetic operator.

\begin{proof}[Proof of Theorem \ref{thm:cl_magnetic}]
  We want to prove inequality \eqref{cond:complexstr} for the classical magnetic operator in order to apply Theorem \ref{thm:complex}.
To this end,
we aim at proving that 
\begin{equation}\label{FU:MAGN}
\Re\big\{ \bar u{\mathcal{N}} u\big\}+|u|\Delta|u|\ge0.
\end{equation}
To check this, we observe\footnote{For an alternative proof based
on fractional arguments, see the forthcoming footnote~\ref{FU:MAGN3}
on page~\pageref{FU:MAGN3}.}
that we can make the computations
in the vicinity of a point~$x$ for which~$|u(x)|>0$.
Indeed, if~\eqref{FU:MAGN} holds true at~$\{|u|>0\}$,
we can fix~$\epsilon>0$ and consider the function~$u_\epsilon:=u+\epsilon$.
In this way, $u_\epsilon(x)=\epsilon>0$, hence we can apply~\eqref{FU:MAGN}
to~$u_\epsilon$ and conclude that
\begin{equation}\label{90:91}
\begin{split}
0 \,&\le \Re\big\{ \bar u_\epsilon(x){\mathcal{N}} u_\epsilon(x)\big\}+|u_\epsilon(x)|\Delta|u_\epsilon(x)|\\
&=
\Re\big\{ (\bar u(x)+\epsilon){\mathcal{N}}u(x)\big\}+
|u(x)+\epsilon|\Delta|u_\epsilon(x)|.
\end{split}\end{equation}
Notice that, for any test function~$\varphi\in C^\infty_0(\Omega)$,
we have that
$$ \lim_{\epsilon\to0}\int_\Omega
\Delta|u_\epsilon(y)|\,\varphi(y)\,dy=
\lim_{\epsilon\to0}\int_\Omega
|u_\epsilon(y)|\,\Delta\varphi(y)\,dy=
\int_\Omega
|u(y)|\,\Delta\varphi(y)\,dy,$$
and so (in the distributional sense)
$$ \lim_{\epsilon\to0} \Delta|u_\epsilon|=
\Delta|u|.$$
Hence, we can pass to the limit in~\eqref{90:91}
and obtain~\eqref{FU:MAGN}.

Accordingly, to prove~\eqref{FU:MAGN}, from now on we will
focus on the case in which~$|u|>0$. We write~$u=a+ib$ and
we observe that 
  \begin{equation}\label{Bvah}
\begin{split}
    &\Re \{ -\bar{u}(\nabla-iA)^2 u   \} \\
  	=\;& \Re \left\{ -\bar{u}(\Delta u - |A|^2 u -iA \cdot \nabla u -\nabla \cdot (iAu) )   \right\} \\
  	=\;& \Re \left\{ -\bar{u} \Delta u + |A|^2 |u|^2 +2\bar{u}iA \cdot \nabla u +i(\nabla \cdot A)|u|^2 \right\} \\
  	 =\;& \Re \left\{ (-a+ib)( \Delta a+i\Delta b)
  	+ |A|^2 (a^2+b^2) +2(b+ia)A \cdot( \nabla a+i\nabla b)
  	+i(\nabla \cdot A)|u|^2 \right\}\\
=\;&  -a\Delta a-b\Delta b
  	+ |A|^2 (a^2+b^2) +2b\nabla a\cdot A -2a\nabla b\cdot A
  	,
  	\end{split}\end{equation}
  where we used the fact that $A$ is
  real valued.
  
On the other hand, at points where~$|u|\ne0$,
\begin{eqnarray*}&& \Delta |u|^2=2|u|\Delta |u|+2|\nabla |u||^2\\ 
{\mbox{and }}&&\nabla |u|= \frac{a \nabla a + b \nabla b}{|u|},
\end{eqnarray*}
therefore
\begin{eqnarray*}
|u|\Delta |u|&=&\frac12\, \Delta |u|^2-|\nabla |u||^2\\
&=& \frac12\,\Delta(a^2+b^2)-\frac{|a \nabla a + b \nabla b|^2}{|u|^2}\\
&=& a\Delta a+b\Delta b+|\nabla a|^2+|\nabla b|^2-\frac{|a \nabla a + b \nabla b|^2}{a^2+b^2}.
\end{eqnarray*}
{F}rom this and~\eqref{Bvah}, we conclude that
\begin{equation}\label{9384-0348}
\begin{split}&
\Re\big\{ \bar u{\mathcal{N}} u\big\}+|u|\Delta|u|
\\ =\;& |\nabla a|^2+|\nabla b|^2-\frac{|a \nabla a + b \nabla b|^2}{a^2+b^2}
+ |A|^2 (a^2+b^2) +2b\nabla a\cdot A -2a\nabla b\cdot A
\\ =\;& \big| aA-\nabla b\big|^2+\big| bA+\nabla a\big|^2
-\frac{|a \nabla a + b \nabla b|^2}{a^2+b^2},
\end{split}\end{equation}
and the latter term is nonnegative, thanks to~\eqref{ST:00}
(applied here with~$t:=A$, $\alpha:=\nabla a$ and~$\beta:=\nabla b$).
This completes the proof of~\eqref{FU:MAGN}.

Then, from~\eqref{FU:MAGN} here and~\cite{SD.EV.VV}
(see in particular the formula before~(2.12) in~\cite{SD.EV.VV},
exploited here with~$p:=2$ and~$m:=2$),
\begin{eqnarray*}
&& \int_\Omega |u|^{s-2}
\Re\big\{ \bar u{\mathcal{N}} u\big\}\,dx\ge
-\int_\Omega |u|^{s-1}\Delta|u|\,dx\\
&&\qquad=\int_\Omega\nabla |u|^{s-1}\cdot\nabla|u|\,dx
\ge C\,\Vert u \Vert_{L^s{(\Omega)}}^{s},
\end{eqnarray*}
for some~$C>0$.
This establishes inequality~\eqref{cond:complexstr} in this case,
with~$\gamma=1$. Hence, Theorem \ref{thm:cl_magnetic} follows
from Theorems~\ref{thm:complex} and~\ref{thm:classic}.
\end{proof}

Now we deal with the fractional magnetic operator.

\begin{proof}[Proof of Theorem \ref{thm:magnetic}]
	We have to verify the structural hypothesis \eqref{cond:complexstr}. We already know
	that the desired inequality holds for the fractional Laplacian $(-\Delta)^{\sigma} v$ for $\sigma \in (0,1)$ and $v\geq 0$ (compare Theorem 1.2 of \cite{SD.EV.VV}). We notice that
	\begin{equation}\label{FU:MAGN2}
	\begin{split}
		\Re& \left\{ \frac{\bar{u}(x,t) \left( u(x,t) - e^{i(x-y)A(\frac{x+y}{2})}u(y,t) \right)}{|x-y|^{n+2\sigma}}  \right\} \\
		& \hspace{3em} = \frac{|{u}(x,t)|^2  - \Re \left\{ e^{i(x-y)A(\frac{x+y}{2})}u(y,t) \bar{u}(x,t )\right\}  }{|x-y|^{n+2\sigma}} \\
		& \hspace{3em} \geq |u(x,t)| \frac{|{u}(x,t)|  -  |u(y,t)| }{|x-y|^{n+2\sigma}}, 
	\end{split}	
	\end{equation}
and therefore\footnote{Interestingly, \label{FU:MAGN3}
integrating and taking the limit as~$\sigma\to1$ in~\eqref{FU:MAGN2},
one obtains an alternative (and conceptually simpler)
proof of~\eqref{FU:MAGN}.
This is a nice example of analysis in a nonlocal setting
which carries useful information to the classical case.}
\begin{equation}\label{8i9ik9iok92}
\int_{\Omega} |u(x,t)|^{s-2} \Re \{ \bar{u}(x,t)\mathcal{N} [u](x,t)\} \; dx 
\geq \int_{\Omega} |u(x,t)|^{s-1} (-\Delta)^{\sigma}|u|(x,t) \; dx.\end{equation}
Also, since $|u|$ is a real and positive function, we can exploit
formula~(2.25) in~\cite{SD.EV.VV} (used here with~$p:=2$) and write that
$$ \int_{\Omega} |u(x,t)|^{s-1} (-\Delta)^{\sigma}|u|(x,t) \; dx\ge
{C} \Vert u \Vert_{L^{s}(\Omega) }^{s}.$$
{F}rom this and~\eqref{8i9ik9iok92} we infer that
condition~\eqref{cond:complexstr} is satisfied in this case
with~$\gamma=1$. 
Then, the desired conclusion in Theorem \ref{thm:magnetic}
follows from Theorems \ref{thm:complex}
and~\ref{thm:classic}.
\end{proof}

\end{document}